\documentclass[a4paper,12pt]{article}
\usepackage{amsmath,amsfonts,amscd,amssymb}
\usepackage[margin=1in]{geometry}
\usepackage[utf8]{inputenc}
\usepackage[active]{srcltx}
\usepackage[T1]{fontenc}
\usepackage{graphicx}
\usepackage{bbm}
\usepackage{url}
\usepackage{color}
\usepackage{stmaryrd}
\usepackage[bf,footnotesize,margin=2cm]{caption}
\usepackage[colorlinks,final]{hyperref}

\newtheorem{theorem}{Theorem}

\newtheorem{corollary}[theorem]{Corollary}
\newtheorem{lemma}[theorem]{Lemma}
\newtheorem{proposition}[theorem]{Proposition}

\newtheorem{remark}[theorem]{Remark}

\numberwithin{equation}{section}

\newenvironment{proof}[1][\relax]%
  {\paragraph{Proof\ifx#1\relax\else~of #1\fi}}%
  {~\hfill$\square$\par\bigskip}

\newcommand{\calC}{\mathcal{C}}
\newcommand{\calD}{\mathcal{D}}
\newcommand{\calE}{\mathcal{E}}
\newcommand{\calF}{\mathcal{F}}

\newcommand{\calV}{\mathcal{V}}

\newcommand{\bbH}{\mathbb{H}}

\newcommand{\bbR}{\mathbb{R}}

\newcommand{\bbT}{\mathbb{T}}

\newcommand{\bbZ}{\mathbb{Z}}

\newcommand{\rk}[1]{\bgroup\color{red}%
  \par\medskip\hrule\smallskip%
  \noindent\textbf{#1}%
  \par\smallskip\hrule\medskip\egroup}

\title{\bf On  the critical  parameters of the  $q\ge4$ random-cluster
  model on isoradial graphs}
\author{V. Beffara \and H. Duminil-Copin \and S. Smirnov}
\date{\today}

\begin{document}

\maketitle

\begin{abstract}
  The critical  surface for  random-cluster model  with cluster-weight
  $q\ge  4$  on isoradial  graphs  is  identified using  parafermionic
  observables. Correlations are also shown to decay exponentially fast
  in  the  subcritical regime.  While  this  result is  restricted  to
  random-cluster models with  $q\ge 4$, it extends  the recent theorem
  of \cite{BD12} to a large class of planar graphs. In particular, the
  anisotropic random-cluster model  on the square lattice  is shown to
  be  critical if  $\frac{p_vp_h}{(1-p_v)(1-p_h)}=q$, where  $p_v$ and
  $p_h$ denote the horizontal  and vertical edge-weights respectively.
  We also mention consequences for Potts models.
\end{abstract}

\section{Introduction}

\subsection{Motivation}

Random-cluster models  are dependent percolation models  introduced by
Fortuin and  Kasteleyn in 1969  \cite{FK72}. They have been an
important  tool  in  the  study  of  phase  transitions.  Among  other
applications, the spin  correlations of Potts models  get rephrased as
cluster    connectivity    properties    of    their    random-cluster
representations,  which allows  for the  use of  geometric techniques,
thus leading to several important applications. Nevertheless, only few
aspects of the random-cluster models are known in full generality.

The   \emph{random-cluster  model}   on  a   finite  connected   graph
$G=(\calV[G],\calE[G])$ is  a model on  edges of this graph,  each one
being either closed  or open. A configuration can be  seen as a random
graph,  whose vertex  set  is $\calV[G]$  and whose  edge  set is  the
collection  of  all  open  edges;  a  \emph{cluster}  is  a  connected
component for this random graph. The probability of a configuration is
proportional to
$$\prod_{e\text{~open}}p_e\prod_{e\text{~closed}}(1-p_e)\cdot q^{\# \;
  \text{clusters}},$$
where    the   \emph{edge-weights}    $p_e\in   [0,1]$    (for   every
$e\in  \calE[G]$) and  the \emph{cluster-weight}  $q\in(0,\infty)$ are
the parameters  of the  model. Extensive literature  exists concerning
these  models; we  refer the  interested  reader to  the monograph  of
Grimmett~\cite{Gri06} and references therein.

For $q\geq 1$,  the model can be extended  to infinite-volume lattices
where  it  exhibits a  phase  transition.  In  general, there  are  no
conjectures for  the value of  the critical surface,  \emph{i.e.}\ the
set of $(p_e)_{e\in \calE[G]}$ for which the model is critical. In the
case  of  planar  graphs,  there  is  a  connection  (related  to  the
Kramers-Wannier  duality  \cite{KW41a,KW41b}   for  the  Ising  model)
between random-cluster models on a graph and on its dual with the same
cluster-weight  $q$  and   appropriately  related  edge-weights.  This
relation  leads, in  the particular  case of  $\mathbb Z^2$  (which is
isomorphic to its  dual) with $p_e=p$ for every $e$,  to the following
natural prediction:  the critical  point $p_c(q)$ is  the same  as the
so-called   \emph{self-dual   point},   which  has   a   known   value
$\sqrt{q}/(1+\sqrt{q})$. This  was proved recently in  \cite{BD12} for
any  $q\ge 1$  (see also  \cite{DI13}). Furthermore,  the size  of the
cluster at the origin was proved  to have exponential decay tail if
$p<p_c(q)$.

The critical  point was  previously known in  three famous  cases. For
$q=1$,  the  model is  simply  Bernoulli  bond-percolation, proved  by
Kesten \cite{Kes80}  to be  critical at  $p_c(1)=1/2$. For  $q=2$, the
self-dual value corresponds  to the critical temperature  of the Ising
model,  as first  derived by  Onsager \cite{Ons44};  one can  actually
couple realizations of the Ising and  random-cluster models to relate the critical
points of each,  see \cite{Gri06} and references  therein for details.
Finally, for  $q\geq 25.72$, a proof  is known based on  the fact that
the random-cluster model exhibits a  first order phase transition; see
\cite{LMMRS91,LMR86}.

A  general challenge in statistical  physics is  to understand the
universal  behavior,  i.e.\  the  behavior of  a  certain  model,  for
instance the planar random-cluster model, on different graphs. A large
class of graphs, which appeared to  be central in different domains of
planar statistical  physics, is the class  of \emph{isoradial graphs}.
An isoradial  graph is a  planar graph  admitting an embedding  in the
plane in such a way that every face is inscribed in a circle of radius
one.  In  such a  case,  we  will say  that  the  embedding itself  is
isoradial.

\begin{figure}[h]
  \begin{center}
    \includegraphics[width=0.50\textwidth]{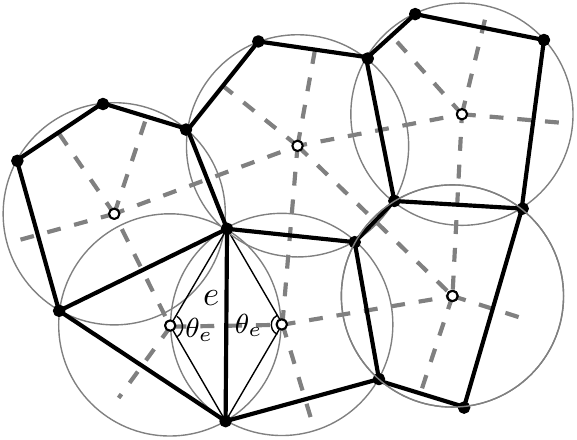}
  \end{center}
  \caption{The black graph is the  isoradial graph. White vertices are
    the  vertices of  the dual  graph. All  faces can  be put  into a
    circumcircle of radius one. Dual  vertices have been drawn in such
    a way that they are the centers of these circles.}
  \label{fig:isoradial}
\end{figure}

Isoradial  graphs were  introduced by  Duffin in  \cite{Duf68} in  the
context  of  discrete  complex   analysis.  The  author  noticed  that
isoradial  embeddings form  a large  class of  embeddings for  which a
discrete notion  of Cauchy-Riemann  equations is  available. Isoradial
graphs first appeared in the physics  literature in the work of Baxter
\cite{Bax78},  where   they  are  called  $Z$-invariant   graphs;  the
so-called  star-triangle transformation  was then  used to  relate the
free energy  of the  eight-vertex and  Ising models  between different
such graphs.  In Baxter's work,  $Z$-invariant graphs are  obtained as
intersections  of lines  in the  plane, and  are not  embedded in  the
isoradial way. The term isoradial was only coined later by Kenyon, who
studied discrete  complex analysis  and the  graph structure  of these
graphs   \cite{Ken02}.  Since   then,  isoradial   graphs  were   used
extensively,   and  we   refer  to   \cite{BTR15,CS08,GM12,Ken02,Mer01}  for
literature on the subject.

In the present article, we study the random-cluster model on isoradial
graphs.

\subsection{Statement of the results}

All the graphs which we will consider in this paper will be assumed to
be \emph{periodic}, in the sense that they will carry an action of the
square lattice  $\bbZ^2$ with  finitely many  orbits; indeed,  this is
often  a crucial  hypothesis  in the  usual  arguments of  statistical
mechanics. Nevertheless, some of our  results extend to a more general
family  of isoradial  graphs,  which  is why  we  first introduce  the
following, weaker condition.

Let $e$  be an edge  of an isoradial  embedding: it subtends  an angle
$\theta_e\in(0,\pi)$ at the center of  the circle corresponding to any
of the  two faces bordered  by $e$; see  Fig.~\ref{fig:isoradial}. Fix
$\theta>0$, and let $G_\infty=(\calV[G_\infty],\calE[G_\infty])$ be an
infinite  isoradial   graph.  The  graph   is  said  to   satisfy  the
\emph{bounded-angle property} if the following condition holds:
\begin{center}
  (BAP$_\theta$) \quad For any $e\in \calE[G_\infty]$,\quad $\theta\le
  \theta_e\le \pi-\theta$.
\end{center}

In order to study the  phase transition, we parametrize random-cluster
measures with cluster-weight  $q\ge 4$ with the help  of an additional
parameter   $\beta>0$.   For   $\beta>0$,   define   the   edge-weight
$p_e(\beta)\in[0,1]$ for $e\in \calE[G_\infty]$ by the formula
$$\frac{p_e(\beta)}{[1-p_e(\beta)]\sqrt q}=\beta
\,\frac{\sinh[\frac{\sigma(\pi-\theta_e)}2]}{\sinh[
  \frac{\sigma\theta_e}2]},$$
where   the   \emph{spin}   $\sigma$   is  given   by   the   relation
$$\cosh \big(\frac  {\sigma\pi} 2\big)  = \frac {\sqrt{q}}  {2}.$$ The
infinite-volume measure  on $G_\infty$  with cluster-weight  $q\ge4 $,
edge-weights  $(p_e(\beta):e\in  \calE[G_\infty])$ and  free  boundary
conditions (see  next section for  a formal definition) is  denoted by
$\phi^0_{G_\infty,\beta,q}$.

\begin{remark}
  In    the    case    of     the    square    lattice,    one    gets
  $p_e(\beta) =  \frac{\beta\sqrt q}{1+\beta\sqrt  q}$. This  does not
  quite  match what  one obtains  in  the setup  of the
  coupling   between  the   Potts  and   random-cluster  models:   the
  bond-parameter corresponding to the $q$-state Potts model at inverse
  temperature $\beta$  is equal to $1-e^{-2\beta}$.  This simply means
  that what we  will denote here by $\beta$ should  not be interpreted
  as  an  inverse temperature  as  such,  but  simply as  a  parameter
  according to which a phase transition can be defined.
\end{remark}

Let $|\cdot|$ be the Euclidean norm.

\begin{theorem}
  \label{exp_decay}
  Let   $q\geq   4$,   $\theta>0$    and   $\beta<1$.   There   exists
  $c=c(\beta,q,\theta)>0$ such  that for any infinite  isoradial graph
  $G_\infty$ satisfying (BAP$_\theta$),
  $$\phi_{G_\infty,\beta,q}^0(u\text{ is connected to } v \text{ by an open path})\leq
  \exp[-c|x-y|],$$for any $u,v\in G_\infty$.
\end{theorem}

This theorem  implies that the edge-weights  $p_e=p_e(1)$ are critical
in the following sense.

\begin{theorem}
  \label{main_theorem}
  Let $q\geq  4$, $\theta>0$. For any  \emph{periodic} isoradial graph
  $G_\infty$:
  \begin{enumerate}
  \item The infinite-volume measure is unique whenever $\beta\ne 1$.
  \item  For  $\beta<1$, there  is  $\phi^0_{G_\infty,\beta,q}$-almost
    surely no infinite-cluster.
  \item  For  $\beta>1$, there  is  $\phi^0_{G_\infty,\beta,q}$-almost
    surely a unique infinite-cluster.
  \end{enumerate}
\end{theorem}

In fact, what  we will prove is the following,  slightly weaker result
in  the  more general  setup  of  graph satisfying  the  bounded-angle
property:

\begin{theorem}
  \label{main_theorem_weak}
  Let $q\geq 4$, $\theta>0$. For any infinite isoradial graph
  $G_\infty$ satisfying (BAP$_\theta$):
  \begin{enumerate}
  \item  For  $\beta<1$, there  is  $\phi^0_{G_\infty,\beta,q}$-almost
    surely no infinite-cluster.
  \item  For  $\beta>1$, there  is  $\phi^1_{G_\infty,\beta,q}$-almost
    surely a unique infinite-cluster.
  \end{enumerate}
\end{theorem}

It will be shown  that in the periodic case, or in  any case for which
the set $\calD_{q,G_\infty}$ of $\beta$  such that there are more than
one  infinite-volume random-cluster  measure  is  of everywhere  dense
complement       (see       Proposition~\ref{uniqueness}       below),
Theorem~\ref{main_theorem_weak}   implies  Theorem~\ref{main_theorem}.
Since this will be the only place where periodicity will be used, most
statements  of this  article  are  phrased (and  proved)  in the  more
general bounded-angle setup.

\bigskip

The theorems  were previously known  for two specific choices  of $q$:
when  $q=2$, the  model was  proved to  be conformally  invariant when
$\beta=1$ in  \cite{CS09}, thus  implying the different  theorems; for
percolation   (i.e.\  the   case   $q=1$),   Manolescu  and   Grimmett
\cite{GM11,Gm11b,GM12}  showed the  corresponding statements  (these papers contain more delicate properties of the critical phase $\beta=1$ as well).
\bigskip

The  main  tools  used  in  the  proofs  are  the so-called \emph{parafermionic
  observables}.   These   observables   were   first   introduced   in
\cite{Smi10} for critical random-cluster  models on $\mathbb Z^2$ with
parameter   $q\in[0,4]$,   as   (anti)-holomorphic   parafermions   of
fractional spin  $\sigma\in[0,1]$, given by certain  vertex operators.
So far,  holomorphicity was  rigorously proved  only for  $q=2$ (which
corresponds to  the Ising model)  and probably holds exactly  only for
this value.  In this case,  the observable  can be used  to understand
many properties  on the model,  including conformal invariance  of the
observable    \cite{CS09,Smi10}    and   loops    \cite{CDHKS12,HK11},
correlations   \cite{CHI12,CI12,Hon10}   and  crossing   probabilities
\cite{BDH12,CDH12,DHN10}. Inspired by  similar considerations, one can
also   compute  the   critical  surface   of  any   bi-periodic  graph
\cite{Li12,CD12}.

Our proof uses an appropriate generalization of these vertex operators
to   random-cluster  models   with  $q\geq   4$.  Even   though  exact
holomorphicity  is not  available, the  observable can  still be  used
efficiently. Interestingly, the spin variable becomes purely imaginary
and does  not possess  an immediate physical  interpretation. However,
this allows us to write better  estimates even in the absence of exact
holomorphicity.   It  also   simplifies  the   relation  between   our
observables and the connectivity properties of the model.

For $\beta\ne  1$, we prove  that the observables behave  like massive
harmonic functions  and decay exponentially  fast with respect  to the
distance to the  boundary of the domain.  Translated into connectivity
properties,  this implies  the sharpness  of the  phase transition  at
$\beta=1$.

The fact that isoradial graphs  are perfect candidate for constructing
parafermionic observables is reminiscent from both the works of Duffin
and Baxter.  Indeed, these works  highlighted the fact  that isoradial
graphs constitute a  general class of graph on  which discrete complex
analysis and statistical physics can be studied precisely.

\paragraph{Application to inhomogeneous models}

The inhomogeneous random-cluster models  on the square, the triangular
and the  hexagonal lattices  can be seen  as random-cluster  models on
periodic   isoradial   graphs.  Theorem~\ref{main_theorem}   therefore
implies the following:

\begin{corollary}
  \label{square}
  The inhomogeneous random-cluster model with cluster-weight $q\ge 4$ on the square, triangular and
  hexagonal lattices $\bbZ^2$, $\bbT$ and $\bbH$ have the following
  critical surfaces:
  \begin{align*}
    &\text{~on~}\bbZ^2\quad \frac{p_1}{1-p_1}\frac{p_2}{1-p_2}=q,\\
    &\text{~on~}\bbT\quad
    \frac{p_1}{1-p_1}\frac{p_2}{1-p_2}\frac{p_3}{1-p_3} +
    \frac{p_1}{1-p_1}\frac{p_2}{1-p_2} +
    \frac{p_1}{1-p_1}\frac{p_3}{1-p_3} +
    \frac{p_2}{1-p_2}\frac{p_3}{1-p_3}=q,\\
    &\text{~on~}\bbH\quad
    \frac{p_1}{1-p_1}\frac{p_2}{1-p_2}\frac{p_3}{1-p_3} =
    q\frac{p_1}{1-p_1}+q\frac{p_2}{1-p_2}+q\frac{p_3}{1-p_3}+q^2,
  \end{align*}
  where $p_1,p_2$ (resp. $p_1,p_2,p_3$) are the edge-weights of the
  different types of edges.
\end{corollary}

For percolation, Corollary~\ref{square} was predicted in \cite{ES63}
and proved in \cite[Section 3.4]{Kes82} for the case of the square
lattice and \cite[Section 11.9]{Gri99} for the case of triangular and
hexagonal lattices.

Let  us also  mention that  the  critical parameter  of the  continuum
random-cluster model  can be computed  using the  fact that it  is the
limit  of inhomogeneous  random-cluster models  on the  square lattice
with  $(p_1,p_2)\rightarrow (0,1)$.  We  refer to  \cite{GOS08} for  a
precise definition of the models, which are connected to the one dimensional Quantum Potts
model.  The parameters  of  the  models are  usually  referred to  as
$\lambda,\delta>0$, where  $\lambda$ and $\delta$ are  the intensities
of the  Poisson Point Process  of bridges and deaths  respectively. In
such case, Theorem~\ref{main_theorem} implies  that the critical point
is given by $\lambda/\delta=q$ for $q\ge 4$.

% , and $\delta$ For each $x\in \bbZ$, consider a Poisson Point
% Process $D_x$ on $\{x\}\times\bbR$ of intensity $\delta$ of death
% points, and a Poisson Point Process $B_x$ of intensity $\lambda$ of
% bridge points. The measure $\phi_{\lambda,\delta,1}$, the
% random-cluster measure of the continuum random-cluster model of
% parameter $\lambda,\delta$ and cluster-weight $1$ is the model of
% random subsets $$\bbZ\times\bbR~\cup~\Big(\bigcup_{x\in \bbZ,y\in
% B_x}[y,y+(1,0)]\Big)~\setminus~ \Big(\bigcup_{x\in \bbZ} D_x\Big)$$

\paragraph{Application to Potts models}

Potts  models   on  $G$  with   $q$  colors  and   coupling  constants
$(J_e:e\in \calE[G])$  can be  coupled with random-cluster  model with
cluster-weight   $q$  and   edge-weights   $p_e=1-\exp[-J_e]$.  As   a
consequence, Theorem~\ref{main_theorem} shows the following:

\begin{corollary}
  \label{Potts}
  Let $q\ge  4$ and  $\theta>0$. For  any infinite  periodic isoradial
  graph  $G_\infty$, the  $q$-state Potts  models on  isoradial graphs
  with   coupling constants  $-\log[1-p_e(1)],   e\in  \calE[G_\infty]$   is
  critical.
\end{corollary}

\subsection{Open questions}

Exact computations  can be performed  for the random-cluster  model at
criticality (see \cite{Bax89}), and despite  the fact that they do not
lead to  fully rigorous mathematical  proofs, they do  provide insight
and further  conjectures on the behavior  of these models at  and near
criticality. Let us mention a few open questions.

\paragraph{1.} Parafermionic observables were  used when $1\le q\le 4$
to prove  that the phase transition  is continuous \cite{Dum12,DST13}.
Moreover,  it  is conjectured  that  among  all random-cluster  models
defined on planar lattices, the phase  transition is of first order if
and only  if $q$ is  greater than 4. Interestingly,  the parafermionic
observable exhibits a very different  behavior for $q\le 4$ and $q>4$,
which raises the following question.

\medbreak  {\bf  Question  1.}  Can  the change  of  behavior  of  the
observable  be related  to  the  change of  critical  behavior of  the
random-cluster model?

\paragraph{2.} In  the work  \cite{BD12}, the  critical value  for the
random-cluster model on the isotropic square lattice has been computed
for any $q\ge  1$. Parafermionic observables on  isoradial graphs also
make  sense for  $q<1$ (see  \cite{Dum12,Smi10}), which  leads to  the
following   question.   \medbreak   \textbf{Question   2.}   Use   the
parafermionic observable  to compute  the critical point  on isoradial
graphs (or simply on $\bbZ^2$) for any $q\in(0,4)$?

\paragraph{3.}  More generally,  parafermionic  observables have  been
found  in  a  number  of   critical  planar  statistical  models,  see
\cite{DS11,Smi10b} and  references therein.  They have  sometimes been
used  to  derive  information  on the  models  (see  \cite{Dum12}  for
random-cluster models and \cite{BBDGG12,BGG12b,BGG12a,DS12,EGGL12} for
$O(n)$-models and  self-avoiding walks). There are many potential applications of these observables which deserve a closer look and we refer to the literature for additional information on these questions.
\paragraph{4.}  As mentioned  earlier,  the  fact that  random-cluster
models on $\bbZ^2$ undergo a first order phase transition is currently
known for $q\geq 25.72$; see \cite{LMMRS91,LMR86}. The main ingredient
is    the    Pirogov-Sinai    theory,    which    shows    that    the
$\phi_{\bbZ^2,1,q}^0$-probability  that  the  origin is  connected  to
distance $n$ decays exponentially fast in $n$. Interestingly, Grimmett
and  Manolescu \cite{GM12}  used the  star-triangle transformation  to
relate  probabilities   of  being   connected  to  distance   $n$  for
percolation  on different  isoradial  graphs.  From \cite{Bax78},  the
star-triangle   transformation  is   known  to   extend  to   critical
random-cluster models  and it seems  plausible that the  techniques in
\cite{GM12}  can be  combined  to results  in \cite{LMMRS91,LMR86}  to
prove  that   the  $\phi^0_{G_\infty,1,q}$-probability  random-cluster
models  that   the  origin  is   connected  to  distance   $n$  decays
exponentially fast in $n$ whenever  $q\ge 25.72$. This would show some
kind of universal behavior: first  order phase transition is common to
any random-cluster model with large enough cluster-weight on isoradial
graphs.  Note  that Pirogov-Sinai  theory  extends  partially to  this
context (although  likely with different  bounds due to the  fact that
the graphs involved would have different combinatorics).

\medbreak \textbf{Question 3.} Show  that random-cluster models on any
isoradial graph  undergo a  first order phase  transition when  $q$ is
large enough.

\paragraph{5.}  Let  us  conclude  with   a  pair  of  more  technical
questions:  How  to  release   the  periodicity  assumption. For instance, how to show
Proposition~\ref{uniqueness} for isoradial  graphs satisfying only the
bounded-angle property? Can the  results be extended to (non-periodic)
isoradial  graphs  which  do   \emph{not}  satisfy  the  bounded-angle
property?

\paragraph{Organization of  the paper.}  Section~\ref{sec:basic} gives
an overview  of probabilistic properties of  the random-cluster model.
It  also introduces  the observable.  Section~\ref{sec:representation}
contains  a derivation  of a  representation formula,  similar to  the
formula for massive harmonic functions,  which is then used to provide
bounds  on the  observable. Section~\ref{sec:proof}  then contain  the
proof  of  Theorem~\ref{exp_decay}  and  Section~\ref{sec:proof2}  the
proofs of Theorem~\ref{main_theorem} and its corollaries.

\section{Basic features of the model}
\label{sec:basic}

We start with an introduction  to the basic features of random-cluster
models.  Details   and  proofs  can   be  found  in   Grimmett's  book
\cite{Gri06}.

\paragraph{Isoradial graphs}

As      mentioned     earlier,      an     \emph{isoradial      graph}
$G=(\calV[G],\calE[G])$ is  a planar  graph admitting an  embedding in
the plane in  such a way that  every face is inscribed in  a circle of
radius one. In such case, we will say that the embedding is isoradial.
For   the   isoradial  embedding,   we   construct   the  dual   graph
$G^*=(\calV[G^*],\calE[G^*])$ as follows:  $\calV[G^*]$ is composed of
all the  centers of  circumcircles of faces  of $G$.  By construction,
every face of  $G$ is associated to a dual  vertex. Then, $\calE[G^*]$
is the  set of edges  between dual vertices corresponding  to adjacent
faces.  Edges of  $\calE[G^*]$ are  in one-to-one  correspondence with
edges of  $\calE[G]$. We  denote the  dual edge  associated to  $e$ by
$e^*$. \medbreak  From now on, we  work only on an  infinite isoradial
graph $G_\infty$ embedded in the isoradial way. Note that the graph is
not a priori periodic.

\paragraph{Definition of the random-cluster model.}

The random-cluster  measure can be  defined on any graph.  However, we
will restrict  ourselves in this  article to the graph  $G_\infty$ and
its connected finite subgraphs.  Let $G=(\calV[G],\calE[G])$ be such a
subgraph.  We  denote by  $\partial  G$  the vertex-boundary  of  $G$,
\emph{i.e.} the  set of sites of  $G$ linked by  an edge to a  site of
$G_\infty\setminus G$.

A \emph{configuration}  $\omega$ on $G$  is a random subgraph  of $G$
having vertex set  $\calV[G]$ and edge set included  in $\calE[G]$. We
will  call the  edges belonging  to $\omega$  \emph{open}, the  others
\emph{closed}. Two sites $u$ and  $v$ are said to be \emph{connected},
if there is an \emph{open path} --- i.e. a path composed of open edges only
---    connecting    them.    The   previous    event    is    denoted by
$u\longleftrightarrow     v$      (we     extend      the     notation
$U\longleftrightarrow V$ to  the event that there exists  an open path
from a vertex of the set $U$ to  a vertex of the set $V$). The connected components of
$\omega$ will be called \emph{clusters}.

A set $\xi$  of \emph{boundary conditions} is given by  a partition of
$\partial G$.  The graph obtained  from the configuration  $\omega$ by
identifying  (or  \emph{wiring}) the  vertices  in  $\partial G$  that
belong to the same component of $\xi$ is denoted by $\omega \cup \xi$.
Boundary conditions  should be  understood as  encoding how  sites are
connected  outside  of  $G$.  Let $k(\omega,\xi)$  be  the  number  of
connected  components  of  $\omega\cup\xi$.  The  probability  measure
$\phi^{\xi}_{G,{\bf  p},q}$ of  the random-cluster  model on  $G$ with
parameters   ${\bf    p}=(p_e:e\in   \calE[G])\in   [0,1]^{\calE[G]}$,
$q\in(0,\infty)$ and boundary conditions $\xi$ is defined by
\begin{equation}
  \label{probconf}
  \phi_{G,{\bf p},q}^{\xi} (\left\{\omega\right\}) =
  \frac {\displaystyle \prod_{e\in \omega}p_e\cdot \prod_{e\notin \omega}(1-p_e)
    \cdot q^{k(\omega,\xi)}}
  {Z_{G,{\bf p},q}^{\xi}},
\end{equation}
for any subgraph  $\omega$ of $G$, where $Z_{G,{\bf  p},q}^{\xi}$ is a
normalizing  constant referred  to as  the \emph{partition  function}.
When there  is no possible  confusion, we  will drop the  reference to
parameters in the notation.

\paragraph{Three specific boundary conditions}

Three boundary conditions will play a special role in our study:
\begin{enumerate}
\item  \emph{Free boundary  conditions}  are  the boundary  conditions
  obtained by  the absence  of wiring  between boundary  vertices. The
  corresponding measure is denoted by $\phi^0_{G,{\bf p},q}$.
\item  \emph{Wired boundary  conditions} are  the boundary  conditions
  obtained  by  wiring  every  boundary  vertices.  The  corresponding
  measure is denoted by $\phi^1_{G,{\bf p},q}$.
\item  Assume  that  $\partial  G$   is  a  self-avoiding  polygon  in
  $G_\infty$, and  let $a$ and $b$  be two sites of  $\partial G$. The
  triple $(G,a,b)$ is called  a \emph{Dobrushin domain}. Orienting its
  boundary  counterclockwise   defines  two  oriented   boundary  arcs
  $\partial_{ab}$  and $\partial_{ba}$;  the \emph{Dobrushin  boundary
    conditions} are defined to be free on $\partial_{ab}$ (there is no
  wiring between  these sites) and  wired on $\partial_{ba}$  (all the
  boundary sites are  wired together). We will refer to  those arcs as
  the free and the wired  arcs, respectively. The measure associated to
  these     boundary     conditions     will     be     denoted     by
  $\phi_{G,{\bf  p},q}^{a,b}$.   We  will  often  use   the  dual  arc
  $\partial^*_{ab}$   adjacent    to   $\partial_{ab}$    instead   of
  $\partial_{ab}$. See Fig.~\ref{fig:diamond_lattice}.
\end{enumerate}

\begin{remark}
  The  term ``Dobrushin boundary condition'' usually refers to mixed $+/-$  boundary condition in  the setup  of the
  Ising model; however the main idea  is the same here, this choice of
  boundary condition  forces the  existence of a  macroscopic boundary
  between two regions  in the domain ($+/-$ for the  Ising model, open
  and dual-open in the case of the random-cluster model), which is why
  we use the same term here.
\end{remark}

\paragraph{The domain Markov property}

One  can  encode, using  appropriate  boundary  conditions $\xi$,  the
influence of the  configuration outside $F$ on the  measure within it.
In  other  words, given  the  state  of  edges  outside a  graph,  the
conditional  measure  inside  $F$  is a  random-cluster  measure  with
boundary conditions  given by the  wiring outside $F$.  More formally,
let $G$ be  a graph and fix  $F\subset \calE[G]$. Let $X$  be a random
variable    measurable   in    terms   of    edges   in    $F$   (call
$\mathcal F_{\calE[G]\setminus  F}$ the $\sigma$-algebra  generated by
edges of $\calE[G]\setminus F$). Then,
$$\phi_{G,{\bf p},q}^\xi(X|\mathcal F_{\calE[G]\setminus
  F})(\psi)=\phi^{\xi\cup\psi}_{F,{\bf                     p},q}(X),$$
where  $\xi$  denotes   boundary  conditions  on  $G$,   $\psi$  is  a
configuration outside $F$  and $\xi\cup \psi$ is  the wiring inherited
from  $\xi$  and  the  edges   in  $\psi$.  We  refer  to  \cite[Lemma
(4.13)]{Gri06} for details.

\paragraph{Comparison between boundary conditions}

Random-cluster models  with parameter  $q\geq 1$  are \emph{positively
  correlated}; see  \cite[Theorem (2.1)]{Gri06}.  It implies  that for
any  boundary  conditions $\psi\leq  \xi$  (meaning  that the  wirings
existing in $\psi$ exist in $\xi$ as well), we have
\begin{equation}
  \label{comparison_between_boundary_conditions}
  \phi^{\psi}_{G,{\bf p},q}(A)\leq \phi^{\xi}_{G,{\bf p},q}(A)
\end{equation}
for   any   increasing  event   $A$.   We   immediately  obtain   that
$\phi^0_{G,{\bf      p},q}(A)\le     \phi^\xi_{G,{\bf      p},q}(A)\le
\phi^1_{G,{\bf                                               p},q}(A)$
for any increasing event $A$ and any boundary conditions $\xi$.

\paragraph{Planar duality}

In two  dimensions, one  can associate  to any  random-cluster measure
with parameters ${\bf p}$ and $q$ on  $G$ a dual measure. Let us focus
on the case of free and wired boundary conditions.

Consider    a   configuration    $\omega$    sampled   according    to
$\phi^0_{G,{\bf p},q}$. Construct an edge  model on $G^*$ by declaring
any  edge  of  the dual  graph  to  be  open  (resp.\ closed)  if  the
corresponding edge of the primal graph is closed (resp.\ open) for the
initial random-cluster model. The new model  on the dual graph is then
a random-cluster measure with wired boundary conditions and parameters
${\bf    p}^*    =    {\bf   p}^*(p,q)\in[0,1]^{E(G^*)}$    and    $q$
satisfying
$$p^*_{e^*} = \frac{(1-p_e)q}{(1-p_e)q+p_e}, \; \text{or equivalently}
\;                      \frac{p^*_{e^*}p_e}{(1-p^*_{e^*})(1-p_e)}=q.$$
This relation  is known as  the \emph{planar duality}.  Similarly, the
dual  boundary  conditions  of  wired  boundary  conditions  are  free
boundary conditions. See \cite[Section 6.1]{Gri06}.

\paragraph{Infinite-volume measures}

A probability  measure $\phi$  on $(\Omega,\mathcal  F)$ is  called an
\emph{infinite-volume  random-cluster}  measure   on  $G_\infty$  with
parameters $p$  and $q$ if for  every event $A\in \mathcal  F$ and any
finite $G\subset G_\infty$,
$$\phi(A|\mathcal F_{\calE[G_\infty\setminus
  G]})(\xi)=\phi^\xi_{G,{\bf                               p},q}(A),$$
for      $\phi$-almost      every     $\xi\in      \Omega$,      where
$\mathcal  F_{\calE[G_\infty\setminus  G]}$  is  the  $\sigma$-algebra
generated by edges in $G_\infty \setminus G$.

The  domain  Markov  property  and  the  comparison  between  boundary
conditions  allow  us to  define  an  infinite-volume measure  as  the
(increasing) limit of a sequence  of random-cluster measures in finite
nested graphs $G_n\nearrow G_\infty$ with free boundary conditions. In
such  cases, the  sequence of  measures is  increasing. We  denote the
corresponding limit measure  $\phi^0_{G_\infty,{\bf p},q}$. Similarly,
one  can  construct  the   measure  $\phi^1_{G_\infty,{\bf  p},q}$  by
considering measures  on nested boxes with  wired boundary conditions. The comparison between boundary conditions~\eqref{comparison_between_boundary_conditions} implies that for any increasing event $A$
\begin{equation}\label{eq:are}\phi^0_{G_\infty,{\bf p},q}(A)\le \phi(A)\le \phi^1_{G_\infty,{\bf p},q}(A).\end{equation}
Section  4 of  \cite{Gri06}  presents a  comprehensive  study of  this
question.

\paragraph{The diamond graph of a Dobrushin domain}

Let   $G_\infty$    be   an    infinite   isoradial    graph.   Define
$G_\infty^\diamond=(\calV[G_\infty^\diamond],\calE[G_\infty^\diamond])$
to        be       the        graph       with        vertex       set
$\calV[G_\infty]\cup \calV[G^*_\infty]$  and edge  set given  by edges
between  a site  $x$  of  $\calV[G_\infty]$ and  a  dual  site $v$  of
$\calV[G^*_\infty]$ if $x$  belongs to the face  corresponding to $v$.
It is  then a \emph{rhombic  graph}, \emph{i.e.}\ a graph  whose faces
are  rhombi;  see  Fig.~\ref{fig:diamond_lattice}.  To  emphasize  the
distinction    with   edges    of    $G_\infty$,   $G^*_\infty$    and
$G^\diamond_\infty$, we refer to the latter as \emph{diamond edges}.

We now define the diamond graph  in the case of Dobrushin domains. Let
$(G,a,b)$    be    a    Dobrushin   domain.    The    diamond    graph
$G^\diamond=(\calV[G^\diamond],\calE[G^\diamond])$ is  the subgraph of
$G_\infty^\diamond$         composed        of         sites        in
$\calV[G]\cup  \calV[G^*]\cup \partial^*_{ab}$  and  of diamond  edges
between them; see Fig.~\ref{fig:diamond_lattice} again.

\paragraph{Loop representation on a Dobrushin domain}

Let $(G,a,b)$ be a Dobrushin domain. In this paragraph, we aim for the
construction of the loop representation of the random-cluster model.

Consider a configuration $\omega$, it defines clusters in $G$ and dual
clusters  in $G^*$.  Through every  face of  the diamond  graph passes
either an open  edge of $G$ or  a dual open edge  of $G^*$. Therefore,
there is a  unique way to draw Eulerian (\emph{i.e.}  using every edge
exactly  once)  loops  on  the diamond  graph  ---  \emph{interfaces},
separating clusters from dual clusters. Namely, loops pass through the
center of  diamond edges, and  in a face  of the diamond  graph, loops
always makes  a turn so  as not  to cross the  open or dual  open edge
through  this face;  see Figure~\ref{fig:diamond_lattice}.  We further
require  that  loops cross  each  diamond  edge orthogonally.  Besides
loops, the  configuration will have  a single curve joining  the edges
adjacent to $a$  and $b$, which are the diamond  edges $e_a$ and $e_b$
connecting   a   site  of   $\partial_{ab}$   to   a  dual   site   of
$\partial_{ba}^*$. This  curve is called the  \emph{exploration path};
we will denote it by $\gamma$. It corresponds to the interface between
the cluster  connected to the  wired arc $\partial_{ba}$ and  the dual
cluster connected to the free arc $\partial_{ab}^*$.

This   provides   us   with   a   bijection   between   random-cluster
configurations   on   $G$   and  Eulerian   loop   configurations   on
$G^{\diamond}$.    This   bijection    is   called    the   \emph{loop
  representation} of  the random-cluster model. We  orientate loops in
such a way that  they cross  every diamond  edge $e$ with the
end-point of $e$ in $\calV[G]$ on its left, and the end-point of $e$ in
$\calV[G^*]$ on its right.

Let  ${\bf  p}\in(0,1)^{\calE[G]}$.  The probability  measure  can  be
nicely  rewritten  (using  Euler's  formula)  in  terms  of  the  loop
picture:
$$\phi_{G,{\bf         p},q}^{a,b}        (\{\omega\})         =\frac1
{\tilde{Z}^{a,b}_{G,{\bf     x},q}}\Big(\prod_{e\in    \omega}x_e\Big)
\sqrt{q}^{\#                    \;                    \text{loops}},$$
where $\tilde{Z}^{a,b}_{G,{\bf  x},q}$ is  a normalizing  constant and
${\bf x}=(x_e:e\in \calE[G])\in(0,\infty)^{\calE[G]}$ is given by
$$x_e=\frac{p_e}{(1-p_e)\sqrt q}.$$

\begin{figure}[ht]
  \begin{center}
    \includegraphics[width=.5\hsize]{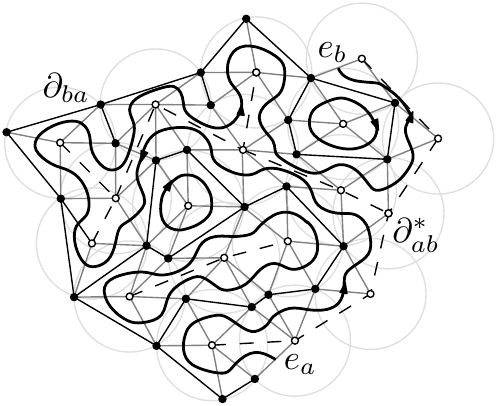}
  \end{center}
  \caption{Construction   of   the   diamond  graph   and   the   loop
    representation.}
  \label{fig:diamond_lattice}
\end{figure}

\paragraph{Critical weights for isoradial graphs}

In the  case of isoradial graphs,  a natural family of  weights can be
defined. Let
$$x_e=\frac{p_e(1)}{(1-p_e(1))\sqrt
  q}=\frac{\sin[\frac{\sigma\theta_e}2]}
{\sin[\frac{\sigma(\pi-\theta_e)}2]}.$$
The  bounded  angle  property  immediately implies  that  weights  are
uniformly bounded away from 0 and 1.

This family of  weights on isoradial graph is self-dual,  in the sense
that   the  dual   of   a  random-cluster   model  with   edge-weights
${\bf  x}=(x_e:e\in  \calE[G])$  is  a random  cluster  model  on  the
(isoradial)         dual        graph         with        edge-weights
$(x_{e^*}:e^*\in \calE[G^*])$.

Let  us stress  out  the  fact that  many  other  families of  weights
${\bf x}$ are self-dual. Nevertheless, this family will play a special
role for reasons that will become apparent later in the article.
%\begin{vb}
%  I agree that duality is written as $x_e^\ast=(x_e)^{-1}$ and that
%  those weights for $\beta=1$ form a self-dual assignment of
%  weights. But for that matter, replacing the ratio of sines on the
%  right by its square leads to a family of weights that is just as
%  self-dual. Hence, at least in the periodic and self-dual case (if
%  such a thing can be non-trivial), it should still be critical. The
%  ``naturality'' of that precise choice is only apparent after
%  introducing the observable.
%\end{vb}

Fix  $\beta>0$.   From  now   on,  we  will   consider  random-cluster
measures
$$\phi_{G,\beta,q}^{\xi}  (\{\omega\}) =\phi_{G,\beta{\bf  x},q}^{\xi}
(\{\omega\})=\frac1                      {\tilde{Z}^{a,b}_{G,\beta{\bf
      x},q}}\Big(\prod_{e\in  \omega}\beta  x_e\Big)  \sqrt{q}^{\#  \;
  \text{loops}},$$
for       any       configuration      $\omega$.       Note       that
$(\phi^0_{G,\beta,q})^*=\phi^1_{G^*,\beta^{-1},q}$.

\paragraph{Phase transition and critical point in the periodic case}

In this paragraph, isoradial graphs are assumed to be periodic. We aim
to study  the behavior  of $\phi^\xi_{G,\beta,q}$ when  $\beta$ varies
from  $0$  to $\infty$.  Positive  association  of the  model  implies
that  $$\phi_{G,\beta,q}^\xi(A)\le \phi^\xi_{G,\beta',q}(A)$$  for any
increasing   event  $A$   and   $\beta\le  \beta'$   (in  such   case,
$\phi_{G,\beta,q}^\xi(A)$ is  said to  be stochastically  dominated by
$\phi_{G,\beta',q}^\xi(A)$).  The previous  inequality extends  to the
infinite volume. It is therefore possible to define
$$\beta_c = \sup\{\beta\ge0:\phi_{G_\infty,\beta,q}^0
(0\longleftrightarrow                                   \infty)=0\},$$
where  $0\longleftrightarrow  \infty$  denotes  the  fact  that  0  is
contained in an infinite open path.  This value is called the critical
point.

We used $\phi_{G_\infty,\beta,q}^0$ to define the critical point but the infinite-volume  measure is not necessarily  unique. Nevertheless,
it can be shown  that for a fixed $q\geq 1$ and a periodic isoradial graph $G_\infty$,  uniqueness can fail only
on a countable set $\mathcal{D}_{q,G_\infty}$. More precisely:

\begin{proposition}\label{uniqueness}
  Let $G_\infty$  be a  periodic isoradial graph.  There exists  an at
  most countable  set $\mathcal{D}_{q,G_\infty}\subset(0,\infty)$ such
  that for  any $\beta\notin \mathcal D_{q,G_\infty}$,  there exists a
  unique infinite-volume measure on $G_\infty$ with parameters $\beta$
  and $q\ge 1$.
\end{proposition}

\begin{proof}
  The proof follows the argument of \cite[Theorem (4.63)]{Gri06} quite
  closely, so we only give a sketch here. Define the \emph{free energy
    per unit volume} in a finite box as
  $$H_G^\xi(\beta)=\frac{1}{|E[G]|}\log \tilde
  Z^\xi_{G,\beta,q}.$$
  We                                                              have
  $$\frac{\partial}{\partial                                    \beta}
  H_G^\xi(\beta)=\frac{1}{|E[G]|}\sum_{e\in  E[G]}
  x_e\phi^\xi_{G,\beta,q}(e\text{     is    open})\in[0,\max\{x_e:e\in
  E[G_\infty]\}];$$
  in  particular,  $H_G^\xi$  is convex.  Now,  let  $G$
  increase  to  cover  the  whole lattice.  A  classical  argument  of
  boundary-area   energy   comparison   (see   the   lines   following
  \cite[Equation (4.71)]{Gri06})  shows that  the limit  $H(\beta)$ of
  $H_G^\xi(\beta)$ exists  and does  not depend on  the boundary
  condition $\xi$.

  Since $H$ is a uniform limit  of convex functions, it is convex, and
  therefore   differentiable  outside   an  at   most  countable   set
  $\mathcal  {D}_{q,G_\infty}$.  We may check  that for     any
  $\beta\notin \mathcal D_{q,G_\infty}$, both
  $\frac{\partial}{\partial    \beta}     H_G^1(\beta)$    and
  $\frac{\partial}{\partial  \beta} H_G^0(\beta)$  converge to
  the     same      limit,     which     is     also      equal     to
  $\frac{\partial}{\partial     \beta}      H(\beta)$. Hence, for any such $\beta$,
  $$\lim_{G\nearrow G_\infty}\frac{1}{|E[G]|}\sum_{e\in
    E[G]}            x_e\phi^0_{G,\beta,q}(e\text{            is
    open})=\lim_{G\nearrow
    G_\infty}\frac{1}{|E[G]|}\sum_{e\in              E[G]}
  x_e\phi^1_{G,\beta,q}(e\text{               is              open})$$
  which is  enough to  guarantee that the  measures $\phi_{G_\infty,\beta,q}^0$
  and $\phi_{G_\infty,\beta,q}^1$ coincide; this  in turn implies uniqueness of
  the Gibbs measure for all $\beta \notin \calD_{q,G_\infty}$.
\end{proof}

Since the infinite-volume  measure is unique for  almost every $\beta$
(at     fixed     $q$),     for    any     infinite-volume     measure
$\phi_{G_\infty,\beta,q}$,
\begin{equation}\label{eq:bre}\phi_{G_\infty,\beta,q}(0\longleftrightarrow \infty)\begin{cases} \
  =0&\text{~if~}\beta<\beta_c\\
  \ >0 &\text{~if~}\beta>\beta_c\end{cases}.\end{equation}

\paragraph{Observables for Dobrushin domains}

Fix a Dobrushin domain $(G,a,b)$  and consider the loop representation
of the random-cluster model.  Following~\cite{Smi10}, we now define an
observable  $F$ on  the  edges  of its  diamond  graph, \emph{i.e.}  a
function  $F   :  \calE[G^{\diamond}]  \to   \mathbb{R}_{+}$.  Roughly
speaking,  $F$  is   a  modification  of  the   probability  that  the
exploration  path passes  through the  center  of an  edge. First,  we
introduce    the    following     definition:    the    \emph{winding}
$\text{W}_{\Gamma}(z,z')$ of  a curve  $\Gamma$ between two  edges $z$
and $z'$ of the diamond graph  is the total rotation (in radians) that
the curve makes from the center of  the diamond edge $z$ to the center
of the diamond edge $z'$.

Let  $q>4$.  We  define  the  observable  $F$  for  any  diamond  edge
$e\in \calE[G^{\diamond}]$ by
\begin{equation}
  \label{defF}
  F(e) = \phi_{G,\beta,q}^{a,b} \left({\rm e}^{\sigma
      \text{W}_{\gamma}(e,e_b)} \mathbbm{1}_{e\in \gamma}\right),
\end{equation}
where $\gamma$ is the exploration path  and $\sigma>0$ is given by the
relation
\begin{equation}
  \cosh \big(\sigma \frac\pi2\big) = \frac {\sqrt{q}} {2}.
\end{equation}
For $\sigma>0$  to exist, $q$ needs  to be larger than  $4$, hence the
hypothesis in the theorem. We define the function $\tilde{F}$ by
\begin{equation}
  \label{defFtilde}
  \tilde F(e) = \phi_{G,\beta,q}^{a,b} \left({\rm e}^{-\sigma
      \text{W}_{\gamma}(e,e_b)} \mathbbm{1}_{e\in \gamma}\right).
\end{equation}

\bigbreak
\begin{remark}
  The                                                       observable
  $G(e)=     \mathbb{E}_{G,a,b}    \left({\rm     e}^{{\rm    i}\sigma
      \text{W}_{\gamma}(e,e_b)}   \mathbbm{1}_{e\in   \gamma}\right)$,
  where  $\sin(\sigma\pi/2)=\sqrt q/2$,  was  introduced  in the  case
  $q\le 4$  in \cite{Smi10} for  the square lattice. When  weights are
  critical, one obtains around each vertex $v$
  $$G(NW)-G(SE)={\rm i}[G(NE)-G(SW)],$$
  where $NW$, $SE$,  $NE$ and $SW$ are the your  edges incident to $v$
  indexed in the obvious way. This  relation can be seen as a discrete
  version of  the Cauchy-Riemann  equation. The  observable is  then a
  holomorphic parafermion  of spin  $\sigma$ (which  itself is  a real
  number in $[0,1]$). For $q\geq  4$, $\sigma$ is purely imaginary and
  does not have  an obvious physical meaning; it  would nonetheless be
  interesting to  find one. In this  article, one could work  with the
  corresponding $G$ for $q>4$, but the definitions in \eqref{defF} and
  \eqref{defFtilde} are easier  to handle for the  applications we have
  in mind.
\end{remark}

\section{A representation formula for the observable}
\label{sec:representation}

Let $(G,a,b)$ be a Dobrushin domain.  In this section, we estimate the
sum of  $F$ over a  set $E\subset \calE[G^\diamond]$ in  various ways.
Let $\calF^\diamond$  be the set  of inner faces of  $G^\diamond$. Any
$f\in    \calF^\diamond$    is    bordered   by    four    edges    in
$\calE[G^\diamond]$, which we label counterclockwise $A$, $B$, $C$ and
$D$, so that  locally the loops (or the exploration  path) go from $f$
to the outside when crossing $A$ and  $C$, and from the outside to $f$
when crossing $B$ and $D$; see Figure~\ref{fig:loop_involution}. There
are a priori two ways to do so, but the choice will be irrelevant.

\begin{lemma}
  \label{relation_around_a_vertex}
  Fix $\beta>0$ and $q>4$. For every face $f\in \calF^{\diamond}$,
  \begin{equation}
    \label{rel_vertex}
    F(B)+F(D) =\Lambda_e(\beta x_e) \left[F(A)+F(C)\right],
  \end{equation}
  where $e$ is the edge of $G$ passing through $f$, and $\Lambda_e$ is
  given                                                             by
  $\displaystyle\Lambda_e(x)                   =                  {\rm
    e}^{-\sigma(\pi-\theta_e)}\frac{x+{\rm                  e}^{\sigma
      \frac\pi2}}{x+{\rm e}^{-\sigma \frac\pi2}}$.
\end{lemma}

A  similar  statement was  used  in  \cite{BDC11a} to  derive  massive
harmonicity of the  observable when $q=2$ on the  square lattice. This
enabled  to compute  the correlation  length of  the high  temperature
Ising model. Observe that $\Lambda(x_e)=1$.

\begin{proof}
  Consider the  involution $s$  on the  space of  configurations which
  switches  the state  (open or  closed) of  the edge  of $G$  passing
  through $f$.

  Let  $e$   be  an  edge   of  the   diamond  graph  and   denote  by
  $e_{\omega}={\rm  e}^{\sigma   W_{\gamma}(e,e_b)}  \mathbbm{1}_{e\in
    \gamma}                                                 p(\omega)$
  the  contribution of  $\omega$ to  $F(e)$ (here  $p(\omega)$ is  the
  probability  of  the  configuration   $\omega$).  Since  $s$  is  an
  involution,        the        following       relation        holds:
  $$F(e)=\sum_{\omega}   e_{\omega}=\textstyle\frac{1}{2}\displaystyle
  \sum_{\omega}     \left[     e_{\omega}+e_{s(\omega)}     \right].$$
  In  order to  prove  (\ref{rel_vertex}), it  suffices  to prove  the
  following for any configuration $\omega$:
  \begin{equation}
    \label{c}
    B_{\omega} +    B_{s(\omega)} + D_{\omega} + D_{s(\omega)}  =
    \Lambda(\beta x_e) \left[A_{\omega} + A_{s(\omega)} + C_{\omega} +
      C_{s(\omega)} \right].
  \end{equation}
  When $\gamma(\omega)$ does  not go through any of  the diamond edges
  bordering   $f$,   neither   does   $\gamma(s(\omega))$.   All   the
  contributions then  vanish and  identity \eqref{c}  trivially holds.
  Thus we may assume that $\gamma(\omega)$ passes through at least one
  edge bordering $f$.  The interface enters $f$ through  either $A$ or
  $C$ and  leaves through $B$ or  $D$. Without loss of  generality, we
  assume that it enters first through $A$ and leaves last through $D$;
  the other cases are treated similarly.

  \begin{figure}[ht]
    \begin{center}
      \includegraphics[width=0.5\hsize]{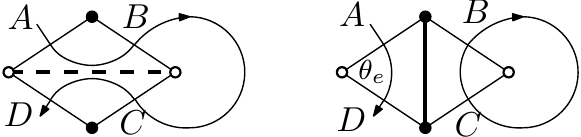}
    \end{center}
    \caption{Two associated configurations $\omega$ and $s(\omega)$}
    \label{fig:loop_involution}
  \end{figure}

  Two cases  can occur (see  Figure~\ref{fig:loop_involution}): Either
  the exploration  curve, after  arriving through $A$,  leaves through
  $B$ and then returns a second time through $C$, leaving through $D$;
  or the exploration curve arrives through $A$ and leaves through $D$,
  with $B$ and $C$ belonging to a loop. Since the involution exchanges
  the two cases, we can assume  that $\omega$ corresponds to the first
  case. Knowing the  term $A_{\omega}$, it is possible  to compute the
  contributions  of  $\omega$ and  $s(\omega)$  to  all of  the  edges
  bordering $f$. Indeed,
  \begin{itemize}
  \item the probability of $s(\omega)$ is equal to $\beta x_e\sqrt{q}$
    times the probability  of $\omega$ (due to the fact  that there is
    one additional loop, and the primal edge crossing $f$ is open);
  \item windings  of the curve can  be expressed using the  winding of
    the  edge   $A$.  For  instance,   the  winding  of  $B$   in  the
    configuration $\omega$  is equal  to the winding  of the  edge $A$
    plus an additional $-\theta_e$ turn.
  \end{itemize}
  Contributions are computed in the following table.
  \bigbreak
  \begin{center}\begin{tabular}{|c|c|c|c|c|}
      \hline
      configuration  &  $A$ &  $B$ &  $C$ & $D$\\
      \hline
      $\omega$ & $A_{\omega}$ &
      ${\rm e}^{\sigma\theta_e}A_{\omega}$ &
      ${\rm e}^{-\sigma\pi}A_{\omega}$ &
      ${\rm e}^{-\sigma(\pi-\theta_e)}A_{\omega}$\\
      \hline
      $s(\omega)$ & $\beta x_e\sqrt{q}A_{\omega}$ & 0 & 0 & ${\rm
        e}^{-\sigma(\pi-\theta_e)}\beta x_e\sqrt{q}A_{\omega}$ \\
      \hline
    \end{tabular}
  \end{center}
  \bigbreak              Using               the              identity
  ${\rm e}^{\sigma\frac\pi2}+{\rm  e}^{-\sigma\frac\pi2}=\sqrt{q}$, we
  deduce  \eqref{c} by  summing  the contributions  of  all the  edges
  bordering $f$.
\end{proof}

For a  set $E$ of  edges of $\calE[G^\diamond]$, $\partial_{\rm  e} E$
denotes the  set of edges of  $\calE[G^\diamond]\setminus E$ bordering
the same face as  an edge of $E$. Also define $E_{\rm  int}$ to be the
set of edges of the diamond graph between two faces of $\calF^\diamond$.

\begin{proposition}
  \label{area-boundary}
  Fix $\beta< 1$ and  $q>4$. Let $G_\infty$ satisfying (BAP$_\theta$).
  There exists $C_1 = C_1(\beta,q,\theta)<\infty$ such that
  \begin{equation*}
    \sum _{e\in E} F(e)\le C_1 \sum_{e\in \partial_{\rm e} E} F(e)
  \end{equation*}
  for any $E\subset E_{\rm int}$.
\end{proposition}

\begin{proof}
  Sum the  identity \eqref{rel_vertex} over  all faces bordered  by an
  edge in $E$. It gives a weighted sum of $F(e)$ (with coefficients
  denoted by $c(e)$) which is identical to zero:
  \begin{equation}
    \label{eq:sum}
    0=\sum_{e\in E} c(e) F(e) +\sum_{e\in \partial_{\rm e} E} c(e)
    F(e).
  \end{equation}

  For an  edge $e\in E$,  $F(e)$ will  appear in two  such identities,
  corresponding  to the  two  faces  it borders.  Since  a loop  going
  through $e$ comes from one of the faces and enters through the other
  one, the  coefficients will be  $-1$ and $\Lambda(\beta  x_e)$. Thus,
  for $e\in E$, $F(e)$ enters the sum with a coefficient
  \begin{align*}
    c(e)=\Lambda(\beta x_e)-1&={\rm
    e}^{-\sigma(\pi-\theta_e)}\left(\frac{\beta x_e+{\rm e}^{\sigma
          \frac\pi2}}{\beta x_e+{\rm e}^{-\sigma
          \frac\pi2}}\right)-1=\left(\frac{x_e+{\rm e}^{-\sigma
          \frac\pi2}}{x_e+{\rm e}^{\sigma \frac\pi2}}\right)\left(\frac{\beta x_e+{\rm e}^{\sigma
          \frac\pi2}}{\beta x_e+{\rm e}^{-\sigma
          \frac\pi2}}\right)-1\\
    % &=\left(\frac{x_e+{\rm e}^{\sigma \frac\pi2}+(\beta
    % -1)x_e}{x_e+{\rm e}^{\sigma \frac\pi2}}\right)\left(\frac{\beta
    % x_e+{\rm e}^{-\sigma \frac\pi2}-(\beta-1)x_e}{\beta x_e+{\rm
    % e}^{-\sigma \frac\pi2}}\right)-1\\
    &=\left(1+\frac{(\beta -1)x_e}{x_e+{\rm e}^{\sigma
          \frac\pi2}}\right)\left(1-\frac{(\beta-1)x_e}{\beta x_e+{\rm
          e}^{-\sigma \frac\pi2}}\right)-1\\
    % &=(\beta -1)x_e\left(\frac{1}{x_e+{\rm e}^{\sigma
    % \frac\pi2}}-\frac{1}{\beta x_e+{\rm e}^{-\sigma
    % \frac\pi2}}-\frac{(\beta-1)x_e}{(x_e+{\rm e}^{\sigma
    % \frac\pi2})(\beta x_e+{\rm e}^{-\sigma \frac\pi2})}\right)\\
    &=(\beta-1)x_e\left(\frac{{\rm e}^{-\sigma \frac\pi2}-{\rm
          e}^{\sigma \frac\pi2}}{(x_e+{\rm e}^{\sigma
          \frac\pi2})(\beta x_e+{\rm e}^{-\sigma
          \frac\pi2})}\right)=\frac{2(1-\beta)
      x_e\sinh(\textstyle\sigma\frac\pi2)}{(x_e+{\rm
        e}^{\sigma \frac\pi2})(\beta x_e+{\rm e}^{-\sigma
        \frac\pi2})}\\
    &\ge  2(1-\beta)\,\min \big\{x_e:e\in
    \calE[G]\big\}\,\sinh(\textstyle\sigma\frac\pi2).
  \end{align*}
  In the  second equality, we  used that $\Lambda(x_e)=1$.  Because of
  the  bounded  angle property,  $x_e$  is  bounded  away from  0  and
  $\infty$ uniformly, and so is $c(e)$.

  For an edge $e\in\partial_{\rm e}  E$, $F(e)$ will appear in exactly
  one identity, (corresponding to the face that it shares with an edge
  of  $E$). The  coefficient  will be  $\Lambda(\beta  x_e)$ or  $-1$,
  depending on the  orientation of $e$ with respect to  the face. Thus
  $F(e)$ will enter the sum with a coefficient $c(e)$ which is bounded
  from above uniformly in $e$  (thanks to the Bounded Angle Property).
  The proposition follows immediately by setting
  $$C_1 =C_1(\beta,q,\theta):= \frac{\max \{|c(e)|:e\in \partial_{\rm e}
    E\}}{\min\{|c(e)|:e\in        E\}}=\frac{\max        \{|c(e)|:e\in
    E[G_\infty^\diamond]\}}{\min\{|c(e)|:e\in
    E[G_\infty^\diamond]\}}<0.$$
  Note that  $C_1$ depends only on $\beta$, $q$  and $\theta$, but  not on
  $G_\infty$ or $E$.
\end{proof}

Write   $\partial   \calE[G^\diamond]$   for  the   edge-boundary   of
$G^\diamond$, meaning the set of diamond edges connecting two boundary
vertices. Note that it is simply $E[G^\diamond]\setminus E_{\rm int}$.

\begin{lemma}
  \label{boundary}
  Let $(G,a,b)$ be a Dobrushin domain. For  a site $u$ on the free arc
  $\partial_{ab}$ and $e\in \partial \calE[G^\diamond]$ a diamond edge
  incident to  $u$, we  have \begin{equation*}  F(e)~=~{\rm e}^{\sigma
      W(e,e_b)}\cdot\phi_{G,\beta,q}^{a,b}(u\leftrightarrow
    \text{wired arc }\partial_{ba} ),
  \end{equation*}
  where $W(e,e_b)$ is the winding of an arbitrary curve on the diamond
  graph from $e$ to $e_b$.
\end{lemma}

\begin{proof}
  Let $u$  be a site of  the free arc $\partial_{ab}$  and recall that
  the  exploration path  is  the interface  between  the open  cluster
  connected to  the wired arc and  the dual open cluster  connected to
  the free arc. Since $u$ belongs to the free arc, $u$ is connected to
  the wired arc if and only if $e$ belongs to the exploration path, so
  that
  $$\phi_{G,\beta,q}^{a,b}(u\leftrightarrow     \text{wired    arc
  }\partial_{ba})=\phi_{G,\beta,q}^{a,b}(e\in               \gamma).$$
  The edge $e$ being on the boundary, the exploration path cannot wind
  around it, so  that the winding of the curve  is deterministic. Call
  this winding $W(e,e_b)$. We deduce from this remark that
  \begin{align*}
    F(e)=\phi_{G,\beta,q}^{a,b} ({\rm e}^{\sigma W_\gamma(e,e_b)}
    \mathbbm{1}_{e\in \gamma}) &= {\rm e}^{\sigma
      W(e,e_b)}\phi_{G,\beta,q}^{a,b}(e\in \gamma)={\rm e}^{\sigma
      W(e,e_b)}\phi_{G,\beta,q}^{a,b}(u\leftrightarrow \text{wired
      arc }\partial_{ba}).
  \end{align*}
\end{proof}

We are now in a position to prove our key proposition.

\begin{proposition}\label{prop:14}
  Fix  $q>4$ and  $G_\infty$ satisfying  (BAP$_\theta$). There  exists
  $C_2=C_2(q,\theta)<\infty$ such that
  $$\sum_{e\in \partial \calE[G^\diamond]}F(e)\le C_2{\rm e}^{\sigma
    (W_{\rm                    max}-W_{\rm                    min})}$$
  for any  $\beta\le 1$  and any  Dobrushin domain  $(G,0,0)$ (here,  $0$ is  assumed  to belong  to $\partial  G$).
  Above, $W_{\rm min}$  and $W_{\rm max}$ are the  minimal and maximal
  winding when going along the boundary of $G$, starting from $0$.
\end{proposition}

Note  that Dobrushin  boundary conditions  on $(G,0,0)$  coincide with
free boundary conditions.

\begin{proof}
  Let   us    start   with   the   case    $\beta=1$.   Sum   Identity
  \eqref{rel_vertex} over all $\calF^\diamond$. Since $\Lambda(x_e)=1$
  for any  $e$, we find that  $c(e)=0$ for any $e\in  E_{\rm int}$ and
  $c(e)=\pm 1$ on $\partial E[G^\diamond]$  depending on the fact that
  a loop going through $e$ points outwards or inward $\calF^\diamond$.
  Boundary edges corresponding  to a loop pointing  outward are called
  \emph{exiting},  those  for  which   the  loop  is  pointing  inward
  \emph{entering}. We find
  $$\sum_{e\text{~exiting}}F(e)~-\sum_{e\text{~entering}}F(e)=0.$$
  Since      edges     exiting      or     entering      belong     to
  $\partial\calE[G^\diamond]$, Lemma~\ref{boundary} implies that
  $$F(e)={\rm e}^{\sigma W(e,e_b)}\phi^0_{G,1,q}
  [u\longleftrightarrow 0],$$
  for  $u$ the  site  of  $G$ bordering  $e$.  Note  that each  vertex
  $u\in \partial G$ is the end-point of a unique entering edge, called
  $e_{\rm in}(u)$ and a unique exiting edge $e_{\rm ex}(u)$. With this
  definition and the two previous displayed equalities, we find
  $$\sum_{u\in \partial G} \left[{\rm e}^{\sigma W(e_{\rm
        ex}(u),e_b)}-{\rm e}^{\sigma W(e_{\rm
        in}(u),e_b)}\right]\phi^0_{G,1,q}(u\longleftrightarrow 0)=0$$
  which can be rewritten as
  $$\sum_{u\in \partial G\setminus\{0\}} \left[{\rm e}^{\sigma
      W(e_{\rm       ex}(u),e_b)}-{\rm       e}^{\sigma       W(e_{\rm
        in}(u),e_b)}\right]\phi^0_{G,1,q}(u\longleftrightarrow 0)={\rm
    e}^{\sigma   W(e_{\rm    in}(0),e_b)}-{\rm   e}^{\sigma   W(e_{\rm
      ex}(0),e_b)}.$$
  Now, when $u=0$, $e_{\rm  in}(0)=e_a$ and $e_{\rm ex}(0)=e_b$. Since
  $W(e_a,e_b)\le 2\pi$,
  $${\rm e}^{\sigma W(e_{\rm in}(0),e_b)}-{\rm e}^{\sigma W(e_{\rm
      ex}(0),e_b)}={\rm    e}^{\sigma   W(e_a,e_b)}-{\rm    e}^{\sigma
    W(e_b,e_b)}\le {\rm e}^{\sigma2\pi}.$$ When $u\ne 0$, we find that
  $$W(e_{\rm ex}(u),e_b)-W(e_{\rm in}(u),e_b)=W(e_{\rm ex}(u),e_{\rm
    in}(u))\ge                               \theta_{[uv]}\ge\theta,$$
  where $v$  is a neighbor  of $u$ outside of  $G$, and $[uv]$  is the
  edge between $u$  and $v$. Note that the existence  of this neighbor
  is guaranteed by the fact that $u\in \partial G$. Next,
  $${\rm e}^{\sigma W(e_{\rm ex}(u),e_b)}-{\rm e}^{\sigma W(e_{\rm
      in}(u),e_b)}=\left[{\rm   e}^{\sigma    W(e_{\rm   ex}(u),e_{\rm
        in}(u))}-1\right]{\rm   e}^{\sigma  W(e_{\rm   in}(u),e_b)}\ge
  [{\rm   e}^{\sigma   \theta}-1]{\rm   e}^{\sigma   W_{\rm   min}}.$$
  Therefore,
  $$[{\rm e}^{\sigma \theta}-1]{\rm e}^{\sigma W_{\rm
      min}}\sum_{u\in             \partial            G\setminus\{0\}}
  \phi^0_{G,1,q}(u\longleftrightarrow  0)\le {\rm  e}^{\sigma 2\pi}.$$
  Finally, observe that
  \begin{align*}
    \sum_{e\in \partial\calE[G^\diamond]}F(e) &=
    \sum_{e\in \partial\calE[G^\diamond]}{\rm
      e}^{\sigma W(e,e_b)}\phi^0_{G,1,q}[u\longleftrightarrow0]\\
    &\le 4{\rm e}^{\sigma W_{\rm max}}\sum_{u\in \partial
      G}\phi^0_{G,1,q}[u\longleftrightarrow0]\le4 {\rm e}^{\sigma
      W_{\rm max}}+4{\rm e}^{\sigma W_{\rm max}}\sum_{u\in \partial
      G\setminus\{0\}}\phi^0_{G,1,q}[u\longleftrightarrow0]\\
    &\le 4{\rm e}^{\sigma W_{\rm max}}+4\frac{{\rm e}^{\sigma (W_{\rm
          max}-W_{\rm min})}}{{\rm e}^{\sigma \theta}-1}{\rm
      e}^{\sigma 2\pi}.
  \end{align*}
  In the  first inequality, we used  the fact that at  most four edges
  correspond to a boundary vertex.

  The     case     $\beta\le      1$     follows     readily     since
  $F(e)={\rm                                                e}^{\sigma
    W(e,e_b)}\phi^0_{G,\beta,q}[u\longleftrightarrow0]$
  is an increasing quantity in $\beta$.
\end{proof}

For  a  graph  $G$,  let   us  introduce  the  following,  recursively
constructed         graphs.         Let        $G^{(0)}=G$         and
$G^{(k)}=G^{(k-1)}\setminus \partial G^{(k-1)}$ for any $k\ge 1$. They
can be seen as successive ``pealings'' of $G$, each step consisting in
removing    the    boundary    of     the    existing    graph.    Let
$E_k=E_{\rm            int}[G^{(k)}]$.            Note            that
$E_0=E_{\rm          int}=\calE[G^\diamond]\setminus          \partial
\calE[G^\diamond]$.

\begin{corollary}
  \label{cor:crucial}
  Let  $G_\infty$ be  an  isoradial  graph satisfying  (BAP$_\theta$).
  Consider the Dobrushin domain $(G,0,0)$ (as above, $0$ is assumed to
  be on the boundary of $G$). For any $\beta<1$ and $q>4$,
  $$\sum_{e\in E_k}F(e)\le C_1C_2{\rm e}^{\sigma (W_{\rm
      max}-W_{\rm min})}\left(\frac{C_1}{1+C_1}\right)^{k}.$$
\end{corollary}

\begin{proof}
  Proposition~\ref{area-boundary}      can      be     applied      to
  $E_k\subset E_{\rm int}$ to give
  $$\sum_{e\in E_k}F(e)\le \frac{C_1}{1+C_1}\sum_{e\in
    E_k\cup\partial_{\rm                                 e}E_k}F(e).$$
  Since $E_k\cup\partial_{\rm e}E_k\subset E_{k-1}$ and $F(e)\ge0$,
  $$\sum_{e\in E_k}F(e)\le\frac{C_1}{1+C_1} \sum_{e\in E_{k-1}}F(e).$$
  Using      the       previous      bound       iteratively,      and
  Proposition~\ref{area-boundary}  one   last  time  (in   the  second
  inequality), we find
  \begin{align*}
    \sum_{e\in E_k}F(e) &
                          \le
                          \left(\frac{C_1}{1+C_1}\right)^{k}\sum_{e\in
                          E_{\rm                           int}}F(e)\le
                          C_1\left(\frac{C_1}{1+C_1}\right)^{k}\sum_{e\in \partial
                          \calE[G^\diamond]}F(e)
  \end{align*}
  The claim follows  by bounding the sum on the  right-hand side using
  Proposition~\ref{prop:14}.
\end{proof}

The study  above can be performed  with $\tilde F$ instead  of $F$. We
obtain the following corollary.

\begin{corollary}
  \label{cor:crucial2}
  Let  $G_\infty$ satisfying  (BAP$_\theta$).  Consider the  Dobrushin
  domain $(G,0,0)$ ($0$ is assumed to be  on the boundary
  of $G$). For any $\beta<1$ and $q>4$,
  $$\sum_{e\in E_k}\tilde{F}(e)\le C_1C_2{\rm e}^{\sigma
    (W_{\rm max}-W_{\rm min})}\left(\frac{C_1}{1+C_1}\right)^{k}.$$
\end{corollary}

\section{Proof of Theorem~\ref{exp_decay}}
\label{sec:proof}

Without loss  of generality, we  assume that $v=0$. Fix  $\beta<1$. We
aim to prove that there exists $c=c(\beta,q,\theta)>0$ such that
$$\phi^0_{G_\infty,\beta,q}(0\longleftrightarrow u)\le \exp(-c|u|)$$
for  any  $u\in  G_\infty$  containing  $0$.  We  now  fix  $G_\infty$
containing 0  and satisfying  (BAP$_\theta$). We  stress out  that the
constants involved  in the proof  depend on  $\theta$ only but  not on
$G_\infty$  or  $u$.  \bigbreak  The case  $q=4$  is  derived  through
stochastic  domination between  random cluster  measures. Indeed,  for
every $\beta<1$,  there exists $(\beta',q)$ with  $q>4$ and $\beta'<1$
such  that  the  random-cluster  measure  $\phi^0_{G_\infty,\beta',q}$
stochastically      dominates      the     random-cluster      measure
$\phi^0_{G_\infty,\beta,4}$. We refer  to \cite[Theorem (3.23)]{Gri06}
for details on  this fact. It follows from  this stochastic domination
that
$$\phi^0_{G_\infty,\beta,4}(0\longleftrightarrow u)\le
\phi^0_{G_\infty,\beta',q}(0\longleftrightarrow                  u)\le
\exp\big[-c(\beta',q,\theta)|u|\big]$$
for any  $u\in G_\infty$.  It is therefore  sufficient to  assume that
$q>4$, which we now do. \medbreak  For a graph $G_\infty$, we identify
a  convex subset  $A$  of $\mathbb  R^2$ with  the  subgraph given  by
vertices  in  $G_\infty\cap  A$  and   edges  between  them.  The  set
$\partial A$  is referring to  $\partial (G_\infty\cap A)$.  Note that
this set is  not necessarily connected, but this will  not be relevant
in the following. For $r>0$, let $B_0(r)=\{x\in \bbR^2:|x|< r\}$.

\begin{lemma}
  \label{lem:proof}
  There     exists     $c_1=c_1(\beta,q,\theta)>0$.    Assume     that
  $v\in G_\infty$ is on the positive real axis. Then,
  $$\phi^0_{[0,\infty)\times(-\infty,\infty),\beta,q}
  (0\longleftrightarrow v)\le \exp[-c_1|v|].$$
\end{lemma}

There could  be no vertex  $v$ on the positive  axis, but there  is no
loss of generality in assuming that $v$ belongs to this positive axis,
since the graph  $G_\infty$ can be rotated around the  origin in order
to obtain an estimate valid for  any vertex $v\in G_\infty$, where the
half-plane  $[0,\infty)\times(-\infty,\infty)$  is   replaced  by  a
half-plane containing $v$ whose boundary  contains 0 and is orthogonal
to the vector given by the coordinates of $v$.

% Observe  that   $G_{m,n}$  is  necessarily  connected   as  soon  as
% $m,n\ge M$ for a certain constant $M=M(\theta)$. Indeed, the bounded
% angle property shows  that the maximum angle  between two successive
% edges incident to the same vertex is smaller than $\pi-\theta$. {\bf
% Un peu plus de precision}

\begin{proof}
  Define  $G_{m,n}$ to  be the  connected component  of the  origin in
  $[0,n]\times[-m,m]$.   Set    $k=|v|/2$.   Let   $m\ge    100$   and
  $n>  \max\{2|v|,100\}$  (these  two  conditions  avoid  nasty  local
  problems).            Apply            Corollaries~\ref{cor:crucial}
  and~\ref{cor:crucial2}  to $G_{m,n}$  and $k$  to find  (we use  the
  notation of the corollary)
  $$\sum_{e\in E_k}F(e)+\tilde F(e)\le 2C_1C_2{\rm
    e}^{\sigma                   (W_{\rm                   max}-W_{\rm
      min})}\left(\frac{C_1}{1+C_1}\right)^{k}.$$
  The  maximum  and  minimum  windings on  $\partial  E[G_{m,n}]$  are
  bounded  by a  certain constant  $C_3<\infty$. This  statement comes
  from the fact  that the winding of $\partial  E[G_{m,n}]$ is roughly
  comparable  of the  winding of  the boundary  of $[0,n]\times[-m,m]$
  (there are a  few local effects to  take care of). Let  us sketch an
  argument.  There exists  a path  of adjacent  faces of  the subgraph
  $H_{m,n}=[0,n]\times[-m,m]\setminus    [4,n-4]\times[-m+4,m-4]$   of
  $G_{m,n}$ going around $[4,n-4]\times[-m+4,m-4]$. The constant 4 has
  been chosen to fit our purpose. It can probably be improved but this
  would   be  of   no   interest  for   the   proof.  In   particular,
  $\partial E[G_{m,n}]$  is contained in $H_{m,n}$.  The bounded angle
  property shows that  two vertices cannot be arbitrary  close to each
  other,  which   prevents  the  existence   of  paths  of   edges  in
  $E[H_{m,n}]$ winding arbitrary often around  a point of $\bbR^2$. As
  a  consequence, the  winding along  $\partial E[G_{m,n}]$  is indeed
  bounded by a universal constant.

  The previous bounds on $W_{\rm max}$ and
  $W_{\rm min}$ imply
  $$\sum_{e\in E_k}F(e)+\tilde F(e)\le
  C_4\left(\frac{C_1}{1+C_1}\right)^{k}.$$ Next, observe that 
  $$F(e)+\tilde F(e)=\phi^0_{G_{m,n},\beta,q}\left[({\rm e}^{\sigma
      W(e,e_b)}+{\rm      e}^{-\sigma      W(e,e_b)})\mathbbm{1}_{e\in
      \gamma}\right]\ge    2\phi^0_{G_{m,n},\beta,q}(e\in   \gamma).$$
  Let   $\mathcal   C_0$   be   the  cluster   of   the   origin   and
  $\partial_{\rm out} \calC_0$ its outer  boundary, meaning the set of
  vertices connected to the free arc by a path in $E[G_\infty]\setminus E[\calC_0]$. A
  diamond  edge $e=[vy]$,  where $v\in  G_{m,n}$ and  $y\in G_{m,n}^*$
  belongs to $\gamma$ if and only if  0 is connected to $v$ and $y$ is
  connected  to the  free arc.  In other  words, $v$  is on  the outer
  boundary of the cluster of 0. We deduce that
  \begin{align*}
    \phi^0_{G_{m,n},\beta,q}(\partial_{\rm out} \calC_0\cap
    G_{m,n}^{(k)}\ne\emptyset)&\le \sum_{u\in
      G_{m,n}^{(k)}}\phi^0_{G_{m,n},\beta,q}(u\in \partial\mathcal
    C_0)\\
    &\le \sum_{e\in E_{\rm int}^{(k-1)}}F(e)+\tilde F(e)\\
    &\le C_4\left(\frac{C_1}{1+C_1}\right)^{k-1}.
  \end{align*}
  Letting $m$ go to infinity and using the uniform bound above,
  $$\phi^0_{G_{\infty,n},\beta,q}(\partial_{\rm out} \calC_0\cap
  G_{\infty,n}^{(k)}\ne                                  \emptyset)\le
  C_5\left(\frac{C_1}{1+C_1}\right)^{k},$$
  where $G_{\infty,n}=[0,n]\times(-\infty,\infty)$. Next, observe that
  $G_{\infty,n}$    does   not    contain    any   infinite    cluster
  $\phi_{G_{\infty,n},\beta,q}$-almost surely  (in fact, this  is true
  for  any $\beta'<\infty$,  since  the graph  is  rough isometric  to
  $\mathbb Z$). Let us provide a rigorous proof of this statement. The
  experienced reader can skip the next paragraph.

  The bounded angle condition implies that the weight $x_e$ is bounded
  from  above   uniformly  on   $G_\infty$.  Therefore,   there  exits
  $c_2=c_2(\beta,q,\theta)>0$ such that for any finite set $S$ of
  dual edges of cardinality $r$
  $$\phi^0_{G_{\infty,n},\beta,q}(\text{every edge in $S$ is
    dual-open}~|~\mathcal   F_{\calE[G_{\infty,n}^*]\setminus   S})\ge
  \exp[-c_2r].$$
  The   constant   $c_2$  comes   from   the   finite-energy  of   the
  random-cluster model, \emph{i.e.}\ the property that the probability
  for an  edge to be  closed is bounded away  from 0 uniformly  in the
  state  of all  the  other edges;  see \cite[Equation  (3.4)]{Gri06}.
  Next, it is possible to divide  the strip into an infinite number of
  finite pieces by considering disconnecting  paths $P$ of length less
  than   $c_3n$  for   some  constant   $c_3=c_3(\theta)<\infty$.  The
  existence of these paths is easily proved by considering a sequence of
  set of  adjacent faces cutting the  strip, and by noticing  that the
  bounded  angle  property  bounds  from above  the  number  of  edges
  bordering a  face by  $\pi/\sin(\theta/2)$. Now, conditioned  on the
  state of the  others paths, each one of these  paths has probability
  larger than $\exp[-c_2c_3n]>0$ of  being dual-open. This immediately
  implies      that     there      is     no      infinite     cluster
  $\phi_{G_{\infty,n},\beta,q}$-almost surely.

  Since there  is no  infinite cluster  almost surely,  $\mathcal C_0$
  intersects     $G_{\infty,n}^{(k)}$      if     and      only     if
  $\partial_{\rm   out}   \calC_0$  intersects   $G_{\infty,n}^{(k)}$.
  Furthermore, edges  have length smaller  than 2, which  implies that
  $u\in G_{\infty,n}^{(k)}$  (we use the  fact that $n-|v|$  and $|v|$
  are larger than $2k$). Hence,
  $$\phi^0_{G_{\infty,n},\beta,q}(0\longleftrightarrow v)\le
  \phi^0_{G_{\infty,n},\beta,q}(\mathcal C_0\cap G_{\infty,n}^{(k)}\ne
  \emptyset)=\phi^0_{G_{\infty,n},\beta,q}(\partial\mathcal    C_0\cap
  G_{\infty,n}^{(k)}\ne                                  \emptyset)\le
  C_5\left(\frac{C_1}{1+C_1}\right)^{k}.$$
  The proof follows by letting $n$ go to infinity and then by choosing
  $c_1=c_1(\beta,q,\theta)>0$ small enough.
\end{proof}

\begin{figure}[h]
  \begin{center}
    \includegraphics[width=0.50\textwidth]{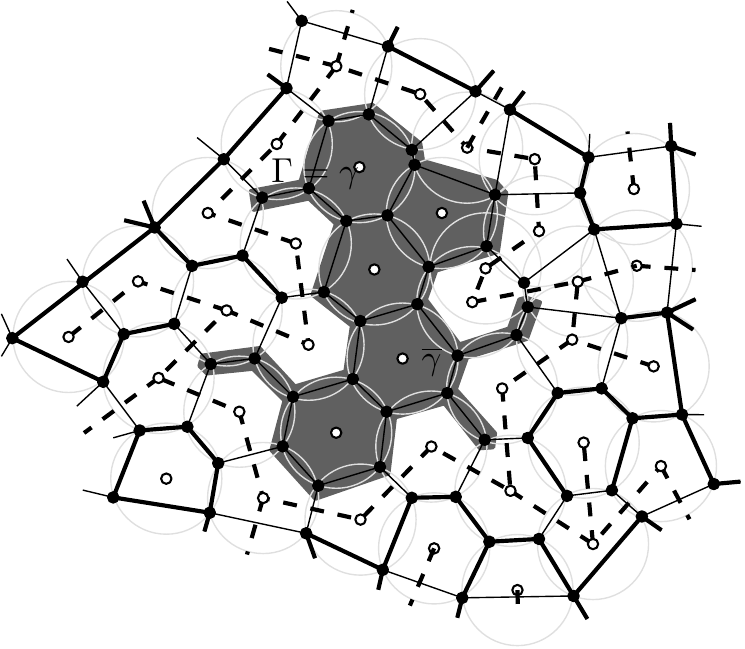}
  \end{center}
  \caption{The  gray  area  is   $\overline{\gamma}$.  The  dual  path
    surrounding it is $\Gamma$. It is the exterior-most dual circuit.}
  \label{fig:isoradial2}
\end{figure}

We  are  now  in  a position  to  prove  Theorem~\ref{exp_decay}.  Let
$n>|u|+2$. We  work with the  random-cluster measure on  $B_0(n)$ with
free  boundary   conditions.  Let  $X_{\rm   max}$  be  the   site  of
$B_0(n)\cap \mathcal  C_0$ which  maximizes its Euclidean  distance to
the origin (when  several such sites exist, take the  first one for an
arbitrary   indexation   of   sites    in   $G_\infty$).   Note   that
$|u|\le |X_{\rm max}|< n$. Therefore,
\begin{align}
  \phi^0_{B_0(n),\beta,q}(0\longleftrightarrow
  u)&\le\phi^0_{B_0(n),\beta,q}(\exists v\in B_0(n)\setminus
  B_0(|u|):X_{\rm max}=v)\nonumber\\
  &\le \sum_{v\in B_0(n)\setminus
    B_0(|u|)}\phi^0_{B_0(n),\beta,q}(X_{\rm
    max}=v)\label{eq:h}.
\end{align}
For $v$, let  $\mathfrak{C}(v)$ be the set  of dual-open self-avoiding
circuits $\gamma$  surrounding the  origin and $v$  and such  that any
site  of  $B_0(n)$  surrounded  by  $\gamma$  is  in  $B_0(|v|)$.  Let
$\overline{\gamma}$  be the  set of  sites of  $B_0(n)$ surrounded  by
$\gamma\in \mathfrak  C(v)$, see  Fig.~\ref{fig:isoradial2}. Dual-open
circuits in $\mathfrak{C}(v)$ are  naturally ordered via the following
order  relation:   $\gamma$  is   more  exterior  than   $\gamma'$  if
$\overline{\gamma}\subset \overline{\gamma}'$.

If $X_{\rm  max}=v$, then $v$ is  connected to $0$ and  there exists a
circuit  in  $\mathfrak{C}(v)$ which  is  dual-open  (simply take  the
boundary  of   $\mathcal  C_0$,   drawn  on   the  dual   graph).  Let
$\{\Gamma=\gamma\}$ be  the event  that $\gamma$ is  the exterior-most
dual-open  circuit   in  $\mathfrak   C(v)$.  With   this  definition,
$X_{\rm max}=v$ if and only if $v$  is connected to 0 and there exists
$\gamma\in \mathfrak C(v)$ such that $\Gamma=\gamma$. Therefore
\begin{align*}
  \phi^0_{B_0(n),\beta,q}(X_{\rm
  max}=v)&=\phi^0_{B_0(n),\beta,q}(0\longleftrightarrow
           v\text{~and~}\exists \gamma\in \mathfrak C(v):\Gamma=\gamma)\\
         &=\sum_{\gamma \in \mathfrak
           C(v)}\phi^0_{B_0(n),\beta,q}(0\longleftrightarrow
           v\text{~and~}\Gamma=\gamma)\\
         &=\sum_{\gamma \in \mathfrak
           C(v)}\phi^0_{B_0(n),\beta,q}(0\longleftrightarrow
           v|\Gamma=\gamma)\phi^0_{B_0(n),\beta,q}(\Gamma=\gamma)
\end{align*}
Since  $\Gamma$ is  the  exterior most  circuit in  $\mathfrak{C}(v)$,
$\{\Gamma=\gamma\}$ is  measurable with respect to  dual-edges outside
$\gamma$.   Furthermore,   edges   of  $\gamma$   are   dual-open   on
$\{\Gamma=\gamma\}$,    see   Fig.~\ref{fig:isoradial2}.    Therefore,
conditioned    on    $\{\Gamma=\gamma\}$,     the    measure    inside
$\overline{\gamma}$  is  a  random-cluster model  with  free  boundary
conditions. Hence,
$$\phi^0_{B_0(n),\beta,q}(0\longleftrightarrow
v|\Gamma=\gamma)=\phi^0_{\overline{\gamma},\beta,q}(0\longleftrightarrow
v).$$
Rotate the graph  $B_0(n)$ in such a  way that $v$ is  on the positive
axis. Let  $H=(-\infty,v]\times\mathbb R$. Observe that  by definition
of  $\mathfrak  C(v)$,  $\overline{\gamma}\subset H$.  The  comparison
between boundary conditions leads to
$$\phi^0_{\overline{\gamma},\beta,q}(0\longleftrightarrow
v)\le\phi^0_{H,\beta,q}(0\longleftrightarrow         v\text{        in
}\overline{\gamma})\le  \phi^0_{H,\beta,q}(0\longleftrightarrow  v)\le
\exp[-c_1|v|]$$ by Lemma~\ref{lem:proof}. This implies
$$\phi^0_{B_0(n),\beta,q}(X_{\rm max}=v)\le\sum_{\gamma \in \mathfrak
  C(y)}         \exp[-c_1|v|]\phi^0_{B_0(n),\beta,q}(\Gamma=\gamma)\le
\exp[-c_1|v|].$$
In  the  second  equality,  we  used   the  fact  that  the  union  of
$\{\Gamma=\gamma\}$ for $\gamma\in \mathfrak  C(v)$ is disjoint. Going
back to \eqref{eq:h}, we find
\begin{align*}
  \phi^0_{B_0(n),\beta,q}(0\longleftrightarrow u)&\le \sum_{v\in
    B_0(n)\setminus B_0(|u|)}\exp[-c_1|v|]\\
  &\le \sum_{v\in G_\infty\setminus B_0(|u|)}\exp[-c_1|v|]\\
  &\le \sum_{k\ge 0}|B_0(|u|+k+1)\setminus
  B_0(|u|+k)|\exp[-c_1(|u|+k)]\\
  &\le \sum_{k\ge
    0}|\textstyle\frac{\pi}{\sin(\theta/2)}(|u|+k+1)|
  \displaystyle\exp[-c_1(|u|+k)]\\
  &\le c_4|u|\exp[-c_1|u|].
\end{align*}
In the fourth inequality,  we used  the fact  that the  number of
sites     in    $B_0(r+1)\setminus     B_0(r)$    is     bounded    by
$\frac{\pi}{\sin(\theta/2)}(r+1)$   because  of   the  bounded   angle
property.

The  proof follows  by  letting $n$  go to  infinity  and by  choosing
$c=c(\beta,q,\theta)>0$ small enough.

\begin{remark}
  The main point of the previous proofs, and of our whole approach, is
  that the observable $F$ behaves like a massive-harmonic function. We
  are not able  to make exact computations though, because  we made no
  particular effort to choose our coupling constants with this goal in
  mind (all that matters here is  that the constant $C_1$ is bounded).
  In a  recent work~\cite{BTR15},  Boutillier, de Tilière  and Raschel
  define a  massive Laplacian operator  on isoradial graphs  for which
  they get the explicit rate  of exponential decay of massive-harmonic
  functions. It  would be  interesting to see  whether a  better tuned
  choice of the  $p_e(\beta)$ can lead to  an exactly massive-harmonic
  observable in their sense, as this would lead to sharper asymptotics
  for the two-point function of the model in that case.
\end{remark}

\section{Proofs of the remaining theorems and corollaries}
\label{sec:proof2}

\begin{proof}[Theorem~\ref{main_theorem_weak}]
  Fix  $\beta<1$.  Let  $G_\infty$  be an  infinite  isoradial  graph.
  Without loss of  generality (simply translate the  graph), we assume
  that $y=0\in G_\infty$. For $r>0$, Theorem~\ref{exp_decay} implies
  \begin{align}
    \phi^0_{G_\infty,\beta,q}(0\longleftrightarrow \partial
    B_0(r))&\le \sum_{u\in \partial
      B_0(r)}\phi^0_{G_\infty,\beta,q}(0\longleftrightarrow
    u)\nonumber\\
    &\le \frac{\pi
      r^2}{4\sin(\theta/2)}\exp\big[-c(\beta,q,\theta)(r-2)\big]
    \label{eq:exp_decay2}.
  \end{align}
  In  the  second   inequality,  we  used  the  fact   that  any  site
  $u\in\partial B_0(r)$ is at distance larger than $r-2$ of the origin
  since any primal  edge is of length smaller than  2 (the diameter of
  circles is  2), and that  $u\in \partial  B_0(r)$ is connected  to a
  site outside $B_0(r)$. We also used the fact that the cardinality of
  $B_0(r)$ is smaller than  $\pi r^2/(4\sin(\theta/2))$ since any edge
  of $G_\infty$ corresponds to a face of $G_\infty^\diamond$ of volume
  larger than $4\sin(\theta/2)$ thanks  to the bounded angle property.
  Letting     $r$    go     to     infinity,     we    obtain     that
  $\phi^0_{G_\infty,\beta,q}(0\longleftrightarrow          \infty)=0$.
  \bigbreak Let  us now consider $\phi^1_{G,\beta,q}$  with $\beta>1$.
  For a  dual vertex $y\in  G^*$, let $A(y)$  be the event  that there
  exists a dual-open  circuit surrounding the origin, i.e.\  a path of
  dual-open edges disconnecting $0$ from infinity in $\bbR^2$. Observe
  that the  dual circuit  must go  to distance  $|y|$. Since  the dual
  model is  a random-cluster model  on the dual isoradial  graph, with
  free    boundary    conditions   and    with    $\beta^*=1/\beta<1$,
  \eqref{eq:exp_decay2} implies
  $$\phi^1_{G_\infty,\beta,q}(A(y))\le \frac{\pi
    |y|^2}{4\sin(\theta/2)}\exp\big[-c(1/\beta,q,\theta)(|y|-2)\big]$$
  for any  $y\in G_\infty^*$.  Borel-Cantelli lemma together  with the
  bound  $|G_\infty^*\cap  B_0(r)|\le\frac{\pi  r^2}{4\sin(\theta/2)}$
  implies that
  $$\phi^1_{G_\infty,\beta,q}(\text{there exist infinitely many }y\in
  G_\infty^*:A(y))=0.$$
  This  immediately  implies that  there  exists  an infinite  cluster
  $\phi^1_{G_\infty,\beta,q}$-almost surely. \end{proof}
  
  \begin{proof}[Theorem~\ref{main_theorem}] Let us now turn
  to  the uniqueness  question.  {\bf  This is  the  only place  where
    \emph{a priori} information on uniqueness is used}. Recall that in
  this     case,     $\mathcal      D_{G_\infty,q}$     defined     in
  Proposition~\ref{uniqueness}        is       countable.        Since
  $\phi^0_{G_\infty,\beta,q}=  \phi^1_{G_\infty,\beta,q}$  outside  of
  the countable set $\mathcal D_{G_\infty,q}$, for any $\beta<1$ there
  exists             $\beta<\beta'<1$            such             that
  $\phi^1_{G_\infty,\beta',q}=\phi^0_{G_\infty,\beta',q}$.  We  deduce
  that
  $$\phi^1_{G_\infty,\beta,q}(0\longleftrightarrow \infty)\le
  \phi^1_{G_\infty,\beta',q}(0\longleftrightarrow
  \infty)=\phi^0_{G_\infty,\beta',q}(0\longleftrightarrow \infty)=0.$$
  From     \cite[Theorem     (5.33)]{Gri06},    we     deduce     that
  $\phi^1_{G_\infty,\beta,q}=\phi^0_{G_\infty,\beta,q}$     for    any
  $\beta<1$. Theorem (5.33) is proved in the case of $\bbZ^d$, but the
  proof extends to the context of periodic graphs mutatis mutandis. By
  duality, the  infinite-volume measure is unique  except possibly for
  $\beta=1$.   This    implies   in    particular   that    there   is
  $\phi^0_{G_\infty,\beta,q}$-almost   surely   an  infinite   cluster
  whenever $\beta>1$.
\end{proof}

\begin{proof}[Corollary~\ref{square}]
  We present the proof in the case of the square lattice, the cases of
  triangular   and    hexagonal   lattices   being   the    same.   If
  $\frac{p_1}{1-p_1}\frac{p_2}{1-p_2}=\beta q$,  then $x_1x_2=1$ where
  $x_i=\frac{p_i}{(1-p_i)\sqrt  {q\beta}}$.  By embedding  the  square
  lattice in such a way that every  face is a rectangle inscribed in a
  circle of radius  1, and the aspect ratio is  given by $x_1/x_2$, we
  obtain an isoradial graph with critical weights $x_1$ and $x_2$. The
  model is  therefore subcritical (respectively supercritical)  if and
  only if $\beta<1$ (respectively $\beta>1$).
\end{proof}

\begin{proof}[Corollary~\ref{Potts}]
 Corollary~\ref{Potts} follows directly from the classical coupling between
  random-cluster models and Potts models.
\end{proof}
\paragraph{Acknowledgments.}

This work has been accomplished during the stay of the first author in
Geneva. The authors were supported  by the ANR grant 10-BLAN-0123, the
EU Marie-Curie RTN  CODY, the ERC AG  CONFRA, the NCCR SwissMap founded by the Swiss NSF, as well as  by the another grant from the Swiss
{NSF}.  The authors  would like  to thank  Geoffrey Grimmett  for many
fruitful discussions.

\bibliographystyle{siam}
\bibliography{bibli}

\def\cprime{$'$}
\begin{thebibliography}{10}

\bibitem{Bax78}
{\sc R.~J. Baxter}, {\em Solvable eight-vertex model on an arbitrary planar
  lattice}, Philos. Trans. Roy. Soc. London Ser. A, 289 (1978), pp.~315--346.

\bibitem{Bax89}
\leavevmode\vrule height 2pt depth -1.6pt width 23pt, {\em Exactly solved
  models in statistical mechanics}, Academic Press Inc. [Harcourt Brace
  Jovanovich Publishers], London, 1989.
\newblock Reprint of the 1982 original.

\bibitem{BBDGG12}
{\sc N.~Beaton, M.~Bousquet-M\'elou, H.~Duminil-Copin, J.~de~Gier, and A.~J.
  Guttmann}, {\em The critical fugacity for surface adsorption of self-avoiding
  walks on the honeycomb lattice is $1+\sqrt{2}$}, Comm. Math. Phys, 326
  (2014), pp.~727--754.

\bibitem{BGG12b}
{\sc N.~Beaton, A.~J. Guttmann, and I.~Jensen}, {\em A numerical adaptation of
  {SAW} identities from the honeycomb to other {2D} lattices}, J. Phys. A:
  Math. Theor., 45 (2012), p.~035201.

\bibitem{BGG12a}
\leavevmode\vrule height 2pt depth -1.6pt width 23pt, {\em Two-dimensional
  self-avoiding walks and polymer adsorption: {C}ritical fugacity estimates},
  J. Phys. A: Math. Theor., 45 (2012), p.~055208.

\bibitem{BD12}
{\sc V.~Beffara and H.~Duminil-Copin}, {\em The self-dual point of the
  two-dimensional random-cluster model is critical for $q\ge1$}, Prob. Theory
  Related Fields, 153 (2012), pp.~511--542.

\bibitem{BDC11a}
\leavevmode\vrule height 2pt depth -1.6pt width 23pt, {\em Smirnov's fermionic
  observable away from criticality}, Ann. Probab., 40 (2012), pp.~2667--2689.

\bibitem{BDH12}
{\sc S.~Benoist, H.~Duminil-Copin, and C.~Hongler}, {\em Conformal invariance
  of crossing probabilities for the {I}sing model with free boundary
  conditions}.
\newblock arXiv:1410.3715, 2014.

\bibitem{BTR15}
{\sc C.~Boutillier, B.~de~Tili\`{e}re, and K.~Raschel}, {\em {The Z-invariant
  massive Laplacian on isoradial graphs}}.
\newblock arXiv:1504.00792v1, 2015.

\bibitem{CDH12}
{\sc D.~Chelkak, H.~Duminil-Copin, and C.~Hongler}, {\em Crossing probabilities
  in topological rectangles for the critical planar {FK}-{I}sing model}.
\newblock arXiv:1312.7785, 2013.

\bibitem{CDHKS12}
{\sc D.~Chelkak, H.~{Duminil-Copin}, C.~Hongler, A.~Kemppainen, and
  S.~Smirnov}, {\em Convergence of {I}sing interfaces to {S}chramm's {SLE}
  curves}, C. R. Acad. Sci. Paris Math., 352 (2014), pp.~157--161.

\bibitem{CHI12}
{\sc D.~Chelkak, C.~Hongler, and K.~Izyurov}, {\em Conformal invariance of spin
  correlations in the planar {I}sing model}, Ann. of Math. (2), 181 (2015),
  pp.~1087--1138.

\bibitem{CI12}
{\sc D.~Chelkak and K.~Izyurov}, {\em Holomorphic spinor observables in the
  critical {I}sing model}, Comm. Math. Phys., 322 (2013), pp.~303--332.

\bibitem{CS08}
{\sc D.~Chelkak and S.~Smirnov}, {\em Discrete complex analysis on isoradial
  graphs}, Adv. in Math., 228 (2011), pp.~1590--1630.

\bibitem{CS09}
\leavevmode\vrule height 2pt depth -1.6pt width 23pt, {\em Universality in the
  {2D} {I}sing model and conformal invariance of fermionic observables},
  Inventiones mathematicae, 189 (2012), pp.~515--580.

\bibitem{CD12}
{\sc D.~Cimasoni and H.~Duminil-Copin}, {\em The critical temperature for the
  {I}sing model on planar doubly periodic graphs}, Electronic Journal of
  Probability, 18 (2013), pp.~1--18.

\bibitem{Duf68}
{\sc R.~J. Duffin}, {\em Potential theory on a rhombic lattice}, J.
  Combinatorial Theory, 5 (1968), pp.~258--272.

\bibitem{Dum12}
{\sc H.~Duminil-Copin}, {\em Divergence of the correlation length for planar
  {FK} percolation with $1\le q\le4$ via parafermionic observables}, J. Phys.
  A: Math. Theor., 45 (2012), p.~494013.

\bibitem{DHN10}
{\sc H.~Duminil-Copin, C.~Hongler, and P.~Nolin}, {\em Connection probabilities
  and {RSW}-type bounds for the two-dimensional {FK} {I}sing model},
  Communications in Pure and Applied Mathematics, 64 (2011), pp.~1165--1198.

\bibitem{DI13}
{\sc H.~Duminil-Copin and I.~Manolescu}, {\em The phase transition of the
  planar random-cluster model with $q \geq 1$ is sharp}.
\newblock arXiv:1409.3748, to appear in PTRF, 2013.

\bibitem{DST13}
{\sc H.~Duminil-Copin, V.~Sidoravicius, and V.~Tassion}, {\em Order of the
  phase transition in planar {P}otts models}, in preparation,  (2013).

\bibitem{DS11}
{\sc H.~Duminil-Copin and S.~Smirnov}, {\em Conformal invariance in lattice
  models.}, in Lecture notes, in Probability and Statistical Physics in Two and
  More Dimensions, D.~Ellwood, C.~Newman, V.~Sidoravicius, and W.~Werner, eds.,
  CMI/AMS -- Clay Mathematics Institute Proceedings, 2011.

\bibitem{DS12}
\leavevmode\vrule height 2pt depth -1.6pt width 23pt, {\em The connective
  constant of the honeycomb lattice equals $\sqrt{2+\sqrt 2}$}, Annals of
  Math., 175 (2012), pp.~1653--1665.

\bibitem{EGGL12}
{\sc A.~Elvey~Price, J.~de~Gier, A.~J. Guttmann, and A.~Lee}, {\em Off-critical
  parafermions and the winding angle distribution of the {O(n)} model}, Journal
  of Physics A: Mathematical and Theoretical, 45 (2012), p.~275002.

\bibitem{FK72}
{\sc C.~M. Fortuin and P.~W. Kasteleyn}, {\em On the random-cluster model. {I}.
  {I}ntroduction and relation to other models}, Physica, 57 (1972),
  pp.~536--564.

\bibitem{Gri99}
{\sc G.~Grimmett}, {\em Percolation}, Springer Verlag, 1999.

\bibitem{Gri06}
{\sc G.~R. Grimmett}, {\em The random-cluster model}, vol.~333 of Grundlehren
  der Mathematischen Wissenschaften [Fundamental Principles of Math. Sciences],
  Springer-Verlag, Berlin, 2006.

\bibitem{GM11}
{\sc G.~R. Grimmett and I.~Manolescu}, {\em Inhomogeneous bond percolation on
  square, triangular and hexagonal lattices}, Ann. Probab., 41 (2013),
  pp.~2990--3025.

\bibitem{Gm11b}
\leavevmode\vrule height 2pt depth -1.6pt width 23pt, {\em Universality for
  bond percolation in two dimensions}, Ann. Probab., 41 (2013), pp.~3261--3283.

\bibitem{GM12}
\leavevmode\vrule height 2pt depth -1.6pt width 23pt, {\em Bond percolation on
  isoradial graphs}, Prob. Theory Relatd Fields, 159 (2014), pp.~273--327.

\bibitem{GOS08}
{\sc G.~R. Grimmett, T.~J. Osborne, and P.~F. Scudo}, {\em Entanglement in the
  quantum {I}sing model}, J. Stat. Phys., 131 (2008), pp.~305--339.

\bibitem{Hon10}
{\sc C.~Hongler}, {\em Conformal invariance of {I}sing model correlations}, PhD
  thesis,  (2010), p.~118.

\bibitem{HK11}
{\sc C.~Hongler and K.~Kyt\"{o}l\"{a}}, {\em {Ising interfaces and free
  boundary conditions}}, J. Am. Math. Soc., 26 (2013), pp.~1107--1189.

\bibitem{Ken02}
{\sc R.~Kenyon}, {\em The {L}aplacian and {D}irac operators on critical planar
  graphs}, Invent. Math., 150 (2002), pp.~409--439.

\bibitem{Kes80}
{\sc H.~Kesten}, {\em The critical probability of bond percolation on the
  square lattice equals {$\frac12$}}, Comm. Math. Phys., 74 (1980), pp.~41--59.

\bibitem{Kes82}
{\sc H.~Kesten}, {\em Percolation theory for mathematicians}, vol.~2 of
  Progress in Probability and Statistics, Birkh{\"a}user, Boston, Mass., 1982.

\bibitem{KW41a}
{\sc H.~Kramers and G.~Wannier}, {\em Statistics of the two-dimensional
  ferromagnet, {I}}, Phys. Rev., 60 (1941), pp.~252--262.

\bibitem{KW41b}
\leavevmode\vrule height 2pt depth -1.6pt width 23pt, {\em Statistics of the
  two-dimensional ferromagnet, {II}}, Phys. Rev., 60 (1941), pp.~263--276.

\bibitem{LMMRS91}
{\sc L.~Laanait, A.~Messager, S.~Miracle-Sol\'e, J.~Ruiz, and S.~Shlosman},
  {\em Interfaces in the {P}otts model. {I}. {P}irogov-{S}inai theory of the
  {F}ortuin-{K}asteleyn representation}, Comm. Math. Phys., 140 (1991),
  pp.~81--91.

\bibitem{LMR86}
{\sc L.~Laanait, A.~Messager, and J.~Ruiz}, {\em Phases coexistence and surface
  tensions for the {P}otts model}, Comm. Math. Phys., 105 (1986), pp.~527--545.

\bibitem{Li12}
{\sc Z.~{Li}}, {\em Critical temperature of periodic {I}sing models}, Comm.
  Math. Phys., 315 (2012), pp.~337--381.

\bibitem{Mer01}
{\sc C.~Mercat}, {\em Discrete {R}iemann surfaces and the {I}sing model}, Comm.
  Math. Phys., 218 (2001), pp.~177--216.

\bibitem{Ons44}
{\sc L.~Onsager}, {\em Crystal statistics. i. a two-dimensional model with an
  order-disorder transition.}, Phys. Rev. (2), 65 (1944), pp.~117--149.

\bibitem{Smi10}
{\sc S.~Smirnov}, {\em Conformal invariance in random cluster models. {I}.
  {H}olomorphic fermions in the {I}sing model}, Ann. of Math. (2), 172 (2010),
  pp.~1435--1467.

\bibitem{Smi10b}
\leavevmode\vrule height 2pt depth -1.6pt width 23pt, {\em Conformal invariance
  in random cluster models. {II}. {S}caling limit of the interface}, 2010.

\bibitem{ES63}
{\sc M.~F. Sykes and J.~W. Essam}, {\em Some exact critical percolation
  probabilities for bond and site problems in two dimensions}, Phys. Rev.
  Lett., 10 (1963), pp.~3--4.

\end{thebibliography}

\begin{flushright}\footnotesize\obeylines
  \textsc{Unit\'e de Math\'ematiques Pures et Appliqu\'ees}
  \textsc{\'Ecole Normale Sup\'erieure de Lyon}
  \textsc{F-69364 Lyon CEDEX 7, France}
  \textsc{E-mail:} \url{Vincent.Beffara@ens-lyon.fr}
  \bigskip
  \textsc{Section de Math\'ematiques}
  \textsc{Universit\'e de Gen\`eve}
  \textsc{Gen\`eve, Switzerland}
  \textsc{E-mails:} \url{hugo.duminil@unige.ch}, %
  \url{stanislav.smirnov@unige.ch}
\end{flushright}

\end{document}